\documentclass[12pt]{amsart}

\usepackage{amssymb, a4wide, mathdots,url,hyperref}
\usepackage[usenames]{color}
\usepackage[all]{xy}

\providecommand{\germ}{\mathfrak}
\DeclareMathOperator*{\tensor}{\otimes}

\newcommand{\supp}{{\rm Supp}}
\newcommand{\Alg}{\Pi}
\newcommand{\Cusp}{{\rm Cusp}}
\newcommand{\Irr}{{\rm Irr}}
\newcommand{\jm}{{\bf r}}

\newcommand{\ip}{{\bf i}}
\newcommand{\ind}{{\rm ind}}
\newcommand{\indset}{\mathfrak{I}}
\newcommand{\length}{\ell}

\newcommand{\C}{\mathbb{C}}
\newcommand{\Z}{\mathbb{Z}}

\newcommand{\N}{\mathbb{N}}

\newcommand{\mult}{\mathfrak{m}}
\newcommand{\multn}{\mathfrak{n}}
\newcommand{\inv}{\theta}
\newcommand{\ob}{\imath}
\newcommand{\Nu}{\mathcal{V}}
\newcommand{\sgm}{\operatorname{Sgm}}

\newcommand{\OO}{\mathcal{O}}              

\newcommand{\Ind}{\operatorname{Ind}}
\newcommand{\K}{\mathcal{K}}

\newcommand{\Hom}{\operatorname{Hom}}

\newcommand{\bs}{\backslash}

\newcommand{\diag}{\operatorname{diag}}

\newcommand{\GL}{\operatorname{GL}}
\newcommand{\Sp}{\operatorname{Sp}}

\newcommand{\cusp}{{\operatorname{cusp}}}
\newcommand{\abs}[1]{\left|{#1}\right|}

\newcommand{\Res}{\operatorname{Res}}

\newcommand{\sm}[4]{{\bigl(\begin{smallmatrix}{#1}&{#2}\\{#3}&{#4}
\end{smallmatrix}\bigr)}}

\newtheorem{theorem}{Theorem}[section]
\newtheorem{lemma}[theorem]{Lemma}
\newtheorem{proposition}[theorem]{Proposition}
\newtheorem{remark}[theorem]{Remark}

\newtheorem{definition}[theorem]{Definition}
\newtheorem{corollary}[theorem]{Corollary}
\newtheorem{example}[theorem]{Example}
\newtheorem{hypothesis}[theorem]{Hypothesis}
\newtheorem*{theorem*}{Theorem}

\title {Klyachko models for ladder representations}
\author{Arnab Mitra}
\address{Department of Mathematics, Technion -- Israel Institute of Technology , Haifa 3200003, Israel}
\email{00.arnab.mitra@gmail.com}
\author{Omer Offen}
\address{Department of Mathematics, Technion -- Israel Institute of Technology , Haifa 3200003, Israel}
\email{offen@tx.technion.ac.il}
\author{Eitan Sayag}
\address{Department of Mathematics, Ben-Gurion University of the Negev ,  P.O.B. 653,
B'eer Sheva 84105, ISRAEL}
\email{eitan.sayag@gmail.com}
\date{\today}
\thanks{Arnab Mitra, partially supported by postdoctoral fellowships funded by the Skirball Foundation via the Center for Advanced Studies in Mathematics at Ben-Gurion University of the Negev and the Department of Mathematics, Technion.}
\thanks{Omer Offen, partially supported by ISF grant No. 1394/12}  \thanks{Eitan Sayag, partially supported by ISF grant No. 1138/10.} 

\begin{document}

\setcounter{tocdepth}{1}

\begin{abstract}

We give a new proof for the existence of Klyachko models for unitary representations of ${\rm GL}_{n}(F)$ over a non-archimedean local field $F$.
Our methods are purely local and are based on studying distinction within the class of ladder representations introduced by Lapid and M\'inguez. 
We classify those ladder representations that are distinguished with respect to Klyachko models. We prove the hereditary property of these models for induced representations from arbitrary finite length representations. Finally, in the other direction and in the context of admissible representations induced from ladder, we study the relation between distinction of the parabolic induction with respect to the symplectic groups and distinction of the inducing data. 


\end{abstract}

\maketitle
\tableofcontents
\section{Introduction}

Let $G$ be a totally disconnected locally compact group and $H$ a closed subgroup. A smooth, complex valued representation $(\pi,V)$ of $G$ is called $H$-distinguished if there exists a non-zero linear form $\ell$ on $V$ such that $\ell(\pi(h)v)=\ell(v)$ for all $h\in H$ and $v\in V$. 

If $\pi$ is irreducible, then such a linear form realizes $\pi$ in a space of functions on $G$, to wit,
\[
\pi\simeq\{g\mapsto \ell(\pi(g^{-1})v): v\in \pi\}\subseteq C^\infty(G/H).
\]

The class of $H$-distinguished representations play an important role in the harmonic analysis of the homogeneous space $G/H$ (see \cite{MR1075727} for instance).
Furthermore, distinguished representations are crucial for the global theory of period integrals of automorphic forms, have applications to the study of special values of $L$-functions and to the description of the image of functorial lifts in the Langlands program.

This paper continues the study of \cite{MR2417789} of distinguished representations of $\GL_n$ over a non-archimedean local field $F$ with respect to Klyachko subgroups.

A Klyachko model is an induced representation $\Ind_{H_{2k,r}}^G(\psi)$ of $G=\GL_n(F)$. Here, $n=2k+r$,
$H_{2k,r}$ is the subgroup of $G$ consisting of matrices of the form $\sm{h}{X}{0}{u}$ where $h\in \Sp_{2k}(F)$, $X\in M_{2k\times r}(F)$, $u$ is an upper-triangular unipotent matrix, and $\psi$ is the character of $H_{2k,r}$ trivially extending a non-degenerate character on the upper-triangular unipotent matrices in $\GL_r(F)$. For a fixed $n$, as $k$ varies, these models `interpolate' between the well known Whittaker model (the case $k=0$) and, if $n$ is even, the sympectic model (the case $k=n/2$).

When $F$ is a finite field, they were introduced by Klyachko in \cite{MR691984}. Together they form a complete model in that case (\cite{MR1129515}), that is, they satisfy 
\[
\oplus_{k=0}^{\lfloor n/2 \rfloor} \Ind_{H_{2k,r}}^G(\psi)=\oplus_{\pi\in \hat{G}}\pi.
\]

Klyachko models over non-archimedean local fields were first studied in \cite{MR1078382} where it was observed, among other things, that some irreducible representations do not imbed into a Klyachko model. In \cite{MR2414223} it is proved that the sum $\oplus_{k=0}^{\lfloor n/2 \rfloor} \Ind_{H_{2k,r}}^G(\psi)$ is  multiplicity free. 
We shall say that an irreducible representation $\pi$ admits a Klyachko model if it can be embedded into $\oplus_{k=0}^{\lfloor n/2 \rfloor} \Ind_{H_{2k,r}}^G(\psi).$

The main result in \cite{MR2417789} prescribes a Klyachko model for any irreducible representation in the unitary dual of $\GL_n(F)$ based on Tadi\'c classification. 
This was achieved by first prescribing a Klyachko model to any Speh representation. For a general unitary representation, a hereditary property of Klyachko models for representations parabolically induced from the Speh representations is applied.

Recently, Lapid and M\'inguez introduced a class of irreducible representations of $\GL_n$ over a non-archimedean local field \cite{MR3163355}. Inspired by their presentation in the Zelevinsky classification scheme, they called them {\it ladder} representations. 
(see \S \ref{sss: ladder}  for the definition). The class of ladder representations contains the Speh representations, the building blocks of the unitary dual (see Tadi\'c classification of the unitary dual of $\GL_n(F)$ \cite{MR870688}).  Thus any irreducible unitarizable representation of $\GL_{n}(F)$ is a product of ladder representations. 


In \cite{MR2515933} we found a connection between the Klyachko model and a partition naturally obtained from the Langlands parameter of a representation. Inspired by this relation, we were led to extend our study of Klyachko models to the entire admissible dual.
The present paper provides a collection of results regarding distinction of  representations of finite length with respect to the Klyachko groups.
In particular, in Theorem \ref{thmC:main} below we classify the distinguished ladder representations in the context of Klyachko models over a non-archimedean local field $F$ of characteristic different then $2$. The special case, when $G=\GL_{2n}(F)$ and $H=\Sp_{2n}(F)$ is described in Theorem \ref{thmB:main} below. These results, together with the Hereditary property established in Thereom \ref{thmA:main} below, recovers, using only local methods, our recipe for the Klyachko model of any representation in the unitary dual of a general linear group.

To help understand the motivation for our results and techniques, we mention two general strategies that one could employ to approach the problem of classifying distinguished admissible representations in the context of a reductive $p$-adic group.  The first strategy is based on Langlands classification, the second based on the notion of imprimitive representations.

We start with the strategy based on the Langlands classification and the notion of standard modules. The smooth dual of $G$ was classified by Langlands in terms of tempered representations of Levi subgroups: Every irreducible smooth representation of $G$ is the unique irreducible quotient of a unique standard module.
Clearly, every non-zero $H$-invariant linear form on an irreducible representation produces such a linear form on its standard module.
A possible strategy for classifying $H$-distinguished representations is based on the Langlands classification: 
\begin{enumerate}
\item   Classify all $H$-distinguished standard modules;
\item   Determine if an $H$-invariant linear form on the standard module 
descends to one on the irreducible quotient. 
\end{enumerate}

To implement the first part of this strategy one can use the geometric lemma of Bernstein and Zelevinsky to analyze distinction of induced representations. 
We refer to \cite{MR2930996} and  \cite{MR3421655} for cases where a complete classification of distinguished standard modules was achieved.

Implementing the second step turns out to be subtler. An $H$-invariant linear form on a standard module will induce such a linear form on some irreducible component.  To determine whether the irreducible quotient admits 
an $H$-invariant linear form is equivalent to determining whether one of the $H$-invariant linear forms on the standard module descends to the irreducible quotient. This problem can be approached by studying the distinction properties of the maximal proper submodule. However, in general, not enough is known about its structure.

The second strategy is based on the concept of imprimitive representation. An irreducible representation of $G$ is called {\it imprimitive} if it is not parabolically induced from any proper parabolic subgroup.
Any irreducible representation is induced from an imprimitive one. 
Thus an approach to the classification problem of $H$-distinction on the smooth dual of $G$ could be: 
\begin{enumerate}
\item Classify all $H$-distinguished imprimitive representations
\item Determine the relation between $H$-distinction and parabolic induction. 
\end{enumerate}

We now focus on the case where $G=\GL_n(F)$ or a product of general linear groups. In this case, the second step might be more accessible.
We further propose to carry this step in two stages: 
\begin{itemize}
\item {\bf Hereditary Property}: Showing that $H$-distinction is compatible with parabolic induction
\item {\bf Purity Lemma}: Showing that an $H$-distinguished representation that is induced from a parabolic subgroup, must be induced from a distinguished representation of the Levi subgroup of that parabolic.  
\end{itemize}


The second strategy is problematic even for $\GL_n$ as the classification of imprimitive representations is an open problem. Nevertheless, within the class of ladder representations, the imprimitive representations are easy to describe and one of the contributions of this work is to pursue this second approach for the problem of distinction with respect to Klyachko models.
Moreover, since the maximal proper subrepresentation of the standard module associated to a ladder representation has a particularly simple description we can implement the steps in the first strategy for these representations. 
This makes distinction problems for the class of irreducible representations parabolically induced from the ladder representations more accessible. 

\subsection{Main Results} 

To simplify our exposition we omit from the introduction, the results whose formulation will require heavy notation. The interested reader should look in the body of the paper for further results of interest.

We state our main results in the form of Theorem \ref{thmA:main}, Theorem  \ref{thmB:main} and Theorem  \ref{thmC:main}. Additionally we will formulate a conditional Theorem \ref{thmD:main}. 

Theorem \ref{thmA:main} concerns the distinction of representations of finite length with respect to Klyachko subgroups while Theorem \ref{thmB:main} (resp.~Theorem \ref{thmC:main}) provides a classification of the distinguished ladder representations with respect to the symplectic (resp.~general Klyachko) subgroup. Theorem  \ref{thmD:main}, conditional on a certain combinatorial assumption (Hypothesis \ref{hyp*}), provides a complete classification of representations induced from ladder representations that are distinguished with respect to the symplectic group. Furthermore, assuming Hypothesis \ref{hyp*}, we provide a necessary consition for a standard module to be distinguished by the symplectic group.

For the sake of notational simplification let us say that a smooth finite length representation $\pi$ of $\GL_{n}(F)$ is $\Sp$-distinguished if $n$ is even and $\pi$ is $\Sp_n(F)$-distinguished.

We also require some of the notation and beautiful results of Zelevinsky \cite{MR584084}. 
We recall that for an irreducible cuspidal representation $\rho$ of $\GL_{n}(F)$ and integers $a\le b$ one considers the segment 
\[
\Delta=\Delta_{\rho}=[\nu^a\rho,\nu^b\rho]=\{\nu^i\rho:i=a,\dots,b\},
\] 
where $\nu(g)=|\det(g)|$ for $g\in\GL_n(F)$. We set $b(\Delta)=a$ for the begining, $e(\Delta)=b$ for the end and $\ell(\Delta)=b-a+1$ for the length of $\Delta$. To $\Delta$ one associates a representation $Z(\Delta)$ and a representation $L(\Delta)$ as follows: $Z(\Delta)$ is the unique irreducible subrepresentation while $L(\Delta)$ is the unique irreducible quotient of the Bernstein-Zelevinsky product $\nu^a\rho \times \dots  \times \nu^b\rho$.
To a multi-set $\mult=\{\Delta_{1},\dots,\Delta_{t}\}$ (a set with possible repetitions) of segments of irreducible cuspidal representations one associates an irreducible representation $Z(\mult)$ and an irreducible representation $L(\mult)$ as follows (See \ref{info: segments}): $Z(\mult)$ is the unique irreducible submodule of the  product $Z(\Delta_1) \times Z(\Delta_2) \times \dots \times Z(\Delta_{t})$ where we have arranged the segments $\Delta \in \mult$ in a standard form (See \ref{sec: standard}).
Analogously, the representation $L(\mult)$ is the unique irreducible quotient of the  standard module $\lambda(\mult)=L(\Delta_1) \times L(\Delta_2) \times \dots \times L(\Delta_{t}).$
The Zelevinsky classification implies that the map $\mult \mapsto Z(\mult)$ is a bijection between the set of such muti-sets of segments and the disjoint union of admissible duals of $\GL_n(F)$ for all $n$,  while the Langlands classification implies that the map $\mult \mapsto L(\mult)$ is a bijection between these sets.

\begin{theorem}\label{thmA:main}
\begin{enumerate}
\item \label{part nec}{\bf A necessary condition for $\Sp$-distinction} (See Proposition \ref{prop: Z dist}):  
If $Z(\mult)$ is $\Sp$-distinguished then $\ell(\Delta)$ is even for all $\Delta\in \mult$.
\item \label{part hered}{\bf Hereditary property for Klyachko models} (See Proposition \ref{prop: herad}): 
Let $\pi_i$ be representations of finite length and $n_i=2k_i+r_i$ be such that $\pi_i$ is $(H_{2k_i,r_i},\psi)$-distinguished for $i=1,\dots,t$. Then $\pi=\pi_1\times\cdots\times\pi_t$ is $(H_{2k,r},\psi)$-distinguished where $k=k_1+\cdots+ k_t$ and $r=r_1+\cdots+r_t$.
\item \label{part cusp line}{\bf Reduction to cuspidal lines} \label{part3}(See Proposition \ref{lem: Kly cusp}): 
Let $\pi_i$ be representations of finite length, $i=1,\dots,t$, such that their cuspidal supports, $\supp(\pi_i)$ and $\supp(\pi_j)$ are totally disjoint for all $i\ne j$ (See \ref{def: cusp supp}). Then $\pi=\pi_1\times\cdots\times\pi_t$ admits a Klyachko model if and only if $\pi_i$ admits a Klyachko model for all $i=1,\dots,t$.  
\end{enumerate}
\end{theorem}

The Zelevinsky classification implies that any irreducible representation $\pi$ of $\GL_n(F)$ can be written as a product $\pi=\pi_{1} \times \dots \times \pi_{t}$ where the cuspidal support $\supp(\pi_{i})$ is contained in a cuspidal line 
and $\supp(\pi_{i})$ , $\supp(\pi_{j})$ are disjoint for $i \ne j.$  This is sometimes called a decomposition of $\pi$ into a product of irreducible {\it rigid} representations. Now, using Theorem \ref{thmA:main} \eqref{part3}, the study of distinction with respect to the Klyachko groups is reduced to the study of distinction within the class of rigid irreducible representations. 

We say that a rigid irreducible representation $\pi=L(\mult)$ is a ladder representation if the multi-set of segments $\mult=\{\Delta_{1},\dots, \Delta_{t}\}$ satisfies the conditions $b(\Delta_{1})> \dots > b(\Delta_{t})$ and  $e(\Delta_{1})> \dots > e(\Delta_{t}).$ 

The next theorems provide the classification of distinguished ladder representations in terms of Langlands classification. We begin with $\Sp$-distinguished representations since the result is easier to formulate.

\begin{theorem}\label{thmB:main}
{\bf $\Sp$-distinguished Ladder representations} (See Theorem \ref{thm: dist lad}):
Let $L(\mult)$ be a ladder representation with $\mult=\{\Delta_1,\dots,\Delta_t\}$. Then $L(\mult)$ is $\Sp$-distinguished if and only if, $t$ is even and $\Delta_{2i-1}=\nu\Delta_{2i}$ for all $i=1,\dots,t/2$.

\end{theorem}
The condition on $\mult$ in Theorem \ref{thmB:main} is equivalent to the existence of a multi-set of segments $\multn$ such that $\mult=\multn+\nu \multn.$
We call such $\mult$ a multi-set of {\it Speh} type (See \ref{def: Sp type}).

For the next theorem we need the notion of right-alignment (See \ref{def: ra seg}). For segments $\Delta=[\nu^{a}\rho, \nu^{b}\rho]$ and $\Delta'=[\nu^{a'}\rho, \nu^{b'}\rho]$ we say that $\Delta'$ is {\it right-aligned} with $\Delta$ and write $\Delta' \vdash \Delta$ if 
$a\ge a'+1$ and $ b= b'+1$.
When $\rho$ is a representation of $\GL_{d}(F)$ we label this relation by the integer $r=d(a-a'-1)$  and write $\Delta'\vdash_r \Delta$.

Our description of ladder representations distinguished with respect to Klyachko groups will be given in two steps. 
We will say that a ladder representation $\pi=L(\mult)$ is a {\it proper ladder} if for all $i=1,\dots, t-1$ we have $e(\Delta_{i+1}) \geq b(\Delta_{i})-1.$ The proper ladder representations are imprimitive and every ladder representation is a product of proper ladders in an essentially unique way. This decomposition into proper ladders is explicit in terms of the underlying multi-set of segments associated to the ladder representation.  

\begin{theorem}\label{thmC:main}
{\bf Ladder representations distinguished with respect to Klyachko groups:}
\begin{enumerate}
\item (See Proposition \ref{prop: prop lad})
Let $L(\mult)$ be a proper ladder representation of $\GL_{n}(F)$ with $\mult=\{\Delta_1,\dots,\Delta_t\}$ and let $n=2k+r$. 

\begin{itemize}
\item If $t$ is even then the representation $L(\mult)$ is $(H_{2k,r},\psi)$-distinguished if and only if $\Delta_{t-2i}\vdash_{r_i}\Delta_{t-2i-1}$ 
for some $r_i$ $(i=0,\dots,t/2-1)$ and $r=r_0+\cdots +r_{t/2-1}$. 

\item If $t$ is odd, let $s$ be such that $L(\Delta_1)$ is an irreducible representation of $\GL_{s}(F).$
The representation $L(\mult)$ is $(H_{2k,r},\psi)$-distinguished if and only if $\Delta_{t-2i}\vdash_{r_i}\Delta_{t-2i-1}$ 
for some $r_i$ $(i=0,\dots,(t-3)/2)$ and $r=r_0+\cdots +r_{(t-3)/2}+s$. 
\end{itemize}
\item (See Theorem \ref{thm: kly lad})
Let $\pi=\pi_{1} \times \dots \times\pi_{t}$ be the decomposition of the ladder representation $\pi$ into proper ladder representations $\pi_{i}$, $i=1, \dots, t.$ Then $\pi$ admits a Klyachko model if and only if the proper ladder representations $\pi_{i}$ admit Klyachko models for all $i=1,\dots,t$. 
\end{enumerate}

\end{theorem}

Our last main result contains in its formulation a certain combinatorial property of multi-sets of segments that we call
Hypothesis \ref{hyp*} (See section \S 8). Roughly speaking, it says that the restrictions imposed by the geometric lemma on $\Sp$-distinction of a standard module $\lambda(\mult)$ imply that $\mult$ is of Speh type.
For more details see Section \ref{into: pfs}.

\begin{theorem}\label{thmD:main}
\begin{enumerate}

\item {\bf On $\Sp$-Distinguished Standard modules} (See Proposition \ref{prop: dist hyp}): 
Suppose $\lambda(\mult)$ is rigid and $\Sp$-distinguished. Assume further that $\mult$ satisfies Hypothesis \ref{hyp*}. Then $\mult$ is of Speh type. 
In particular, if $L(\mult)$ is $\Sp$-distinguished and $\mult$ satisfies Hypothesis \ref{hyp*} then $\mult$ is of Speh type.
\item \label{part prod lad}{\bf Distinction for irreducible products of ladder representations} (See Proposition \ref{prop: li}):
Assume Hypothesis \ref{hyp**} holds true for all multisegement. 
Let $\pi_1,\dots,\pi_k$ be ladder representations such that $\pi=\pi_1\times\cdots\times\pi_k$ is an irreducible representation. 
If $\pi$ is $\Sp$-distinguished then $\pi_i$ is $\Sp$-distinguished for all $i=1,\dots,k$.
\item \label{part: pure}{\bf Purity of symplectic distinction within ladder class} (See Corollary \ref{cor: prd two}): 
Let $\pi_1$ and $\pi_2$ be ladder representations such that $\pi=\pi_1\times\pi_2$ is irreducible. 
If $\pi$ is $\Sp$-distinguished then $\pi_1$ and $\pi_2$ are $\Sp$-distinguished.
\end{enumerate}
\end{theorem}
We emphasize that Theorem \ref{thmD:main} \eqref{part: pure} is unconditional. We further show in Proposition \ref{prop: hyp set} that Hypothesis \ref{hyp*} is satisfied by multi-sets that are in fact sets.
This yields the following unconditional result. 
\begin{theorem}(See Corollary \ref{cor: set case})
Let $\lambda(\mult)=L(\Delta_1)\times\cdots\times L(\Delta_t)$ be  a standard module such that $\Delta_i\ne \Delta_j$ for all $i\ne j$. If $\lambda(\mult)$ is $\Sp$-distinguished then $\mult$ is of Speh type.
\end{theorem}

\subsection{Proofs and Methods}\label{into: pfs}

Let us now elaborate a bit on the techniques used in the proofs. The filtration of the geometric lemma allows us to study $\Sp$-distinction of induced representations from the parabolic subgroup $P$ in terms of the geometry of $P$-orbits on the symmetric space $\GL_{2n}(F)/\Sp_{2n}(F).$ In particular we show that an induced $\Sp$-distinguished representation admits a $P$-orbit which is {\it relevant}. Analyzing the relevant orbits together with the Jacquet module calculations of segment representations allows us to prove Theorem \ref{thmA:main} \eqref{part nec}. For Theorem \ref{thmA:main} \eqref{part hered}, we combine the theory of derivatives with a meromorphic continuation technique of Blanc and Delorme. The first is used to reduce the problem to the case of $\Sp$-distinction and the second to construct $\Sp$-invariant linear forms on families of induced representations.

A key ingredient for the proof of  Theorems \ref{thmB:main} and \ref{thmD:main}  is the necessary condition for a standard module to be $\Sp$-distinguished provided by the geometric lemma. 
This allows us to reduce the problem to a purely combinatorial one on multi-sets of segments. We address it under a technical hypothesis that we can prove only for certain multi-sets (in particular, whenever they are sets). The hypothesis can be interpreted as a statement that certain orbits, of the natural action of a parabolic subgroup $P$ of $\GL_{2n}(F)$ on $\GL_{2n}(F)/\Sp_{2n}(F)$, do not contribute a non-trivial $\Sp_{2n}(F)$-invariant linear form. Explicitly, assuming the hypothesis we show that if a representation, irreducibly induced from ladder representations, is $\Sp$-distinguished then each ladder representation in the inducing data also admits a symplectic model.  

This partial result along with the description of the maximal proper subrepresentation of the standard module associated to a ladder representation allows us to finish the 
proof  of  Theorem \ref{thmB:main}.

The proof of  Theorem \ref{thmC:main} is obtained by using the explicit knowledge of the structure of a Jacquet module of a ladder representation, and the classification of $\Sp$-distinguished ladders given by Theorem \ref{thmB:main}.

To prove Theorem \ref{thmD:main} \eqref{part prod lad} we use a recent irreducibility result of \cite{1411.6310}, as well as the invariance of the class of ladder representations with respect to the Zelevinsky involution, to construct an inductive set-up in the context of representations irreducibly induced from ladder representations. We hope that the techniques developed in doing so will be useful more generally. Most notably, when we attempt to study distinction by other closed subgroups of $\GL_{n}(F)$ for representations that are irreducibly induced from ladders.

\subsection{Related Works}

We mention here a few works where the results or the tools used have some intersection with the present work. 

The present work began as an attempt to extend the distinction results of \cite{MR2332593} and \cite{MR2417789} from the unitary dual of $\GL_{n}(F)$ to the admissible dual of $\GL_{n}(F).$ 

We emphasize that in \cite{MR2332593} and \cite{MR2417789} obtaining a model for a Speh representation, in particular the classification of Speh representations admitting a symplectic model, was based on the global theory of period integrals of Eisenstein series and their residues obtained in \cite{MR2254544} and \cite{MR2248833}. A novel aspect of this work is that our method of proof is purely local, and therefore, independently provides a local proof for the results of the aforementioned works.
The methods employed here are very different from those works and are, in fact, closer in spirit to the techniques of \cite{MR1078382}, or to that of \cite{MR3227442} which studies admissible representations distinguished with respect to a symplectic group in small rank cases. 

The focus on distinction problems within the class of ladder representations 
was made in \cite{MR2930996} and later in \cite{MR3421655} and Theorem \ref{thmB:main} of the introduction could be considered as an analogue to their results.
A study parallel to our study of the distinction problem for standard modules can be found in these two references. 

In \cite{MR2889169} the existence of Klyachko models is proved for unitary representations of $\GL_{n}(\mathbb{R})$ and $\GL_{n}(\mathbb{C}).$ The methods there are parallel to those in the works of the second and third author in the non-archimedean case. In particular, the proof of existence of those models for Speh representations was based on the theory of periods of automorphic forms. Recently, in \cite{MR3416438} a local construction of these invariant functionals is provided, based on tools from the theory of distributions and $D$-modules. Some of the results of the present work can be considered as a non-archimedean analogue of the main result of \cite{MR3416438}.

\subsection{Structure of the Paper} 

Let us now delineate the contents of this paper. A large part of it (\S \ref{notn}-\ref{s: dist ladder}) concerns  symplectic models. 

After setting up the general notation for this work in \S \ref{notn}, we recall some well known results concerning $\GL_{2n}(F)/\Sp_{2n}(F)$ in \S \ref{s: sym sp}, especially the structure of orbits of the natural action of a parabolic subgroup of $\GL_{2n}(F)$. Our main tool for studying $\Sp$-distinction of induced representations is an application of the geometric lemma of Bernstein and Zelevinsky. This is recalled in \S \ref{sec: gl}. 

In \S \ref{ss: oc} we obtain some immediate consequences for $\Sp$-distinction of certain induced representations. They come from contributions to the open and to the closed orbits of the aforementioned action. In \S \ref{ss: cusp lines} we reduce the classification of $\Sp$-distinguished irreducible representations to those supported in a single cuspidal line viz. the rigid representations. 

In order to study distinction for rigid representations, we recall in \S \ref{s: LZ cl} the segment notation of Zelevinsky and the classification of the admissible dual. In \S \ref{s: nec Z} we provide a necessary condition for an irreducible representation to be $\Sp$-distinguished in terms of the Zelevinsky classification (Proposition \ref{prop: Z dist}). 

We then turn to the study of distinction of standard modules. A necessary condition for a standard module (and for an irreducible representation) to be $\Sp$-distinguished is reduced in \S \ref{s: std mod sp} to a combinatorial problem. This problem is formulated in \S \ref{s: multisets} as Hypothesis \ref{hyp*}. 

Section \ref{s: multisets} is written in a way completely independent from the rest of the paper, is accessible to any mathematician, and presents a problem with applications to the study of $\Sp$-distinction. 

Our partial results suffice in order to obtain a complete classification of $\Sp$-distinguished ladder representations. This is Theorem \ref{thm: dist lad}.

In \S \ref{s: dist ladder} we obtain our results on $\Sp$-distinction for the class of representations irreducibly induced from ladder. These are  conditional on Hypothesis \ref{hyp*}. Again these results require some purely combinatorial lemmas that we prove in \S \ref{s: ladder sums}.

In \S \ref{s: Kl} we turn to the study of Klyachko models in general. We prove the hereditary property with respect to parabolic induction (Proposition \ref{prop: herad}) and reduce the problem to rigid representations Proposition \ref{lem: Kly cusp}). Finally the classification of ladder representation with a given Klyachko model is obtained in \S \ref{s: kl ladder}.

\subsection{Acknowledgments} 

The authors are grateful to Erez Lapid for sharing with them his insights on ladder representations and irreducibility results.

\section{Notation and preliminaries}\label{notn}

We set the general notation in this section. More particular notation is defined in the section where it first occurs.

\subsection{Generalities}
Let $G$ be a totally disconnected, locally compact group. 
\subsubsection{}
Let $\delta_G$ be the modulus function of $G$ with the convention that $\delta_G(g)dg$ is a right-invariant Haar measure if $dg$ is a left-invariant Haar measure on $G$.

\subsubsection{}
Let $H$ be a closed subgroup of $G$ and $\sigma$ a smooth, complex-valued representation of $H$.
We denote by $\Ind_H^G(\sigma)$ the normalized induced representation. It is the representation of $G$ by right translations on the space of functions $f$ from $G$ to the space of $\sigma$ satisfying 
\[
f(hg)=(\delta_H^{1/2}\delta_G^{-1/2})(h)\sigma(h)f(g), \ \ \ h\in H,\,g\in G
\] 
and $f$ is right invariant by some open subgroup of $G$.
The representation of $G$ on the subspace of functions with compact support modulo $H$ is denoted by $\ind_H^G(\sigma)$.

\subsubsection{}
This paper is concerned with distinguished representations in the following sense.
\begin{definition}
Let $\pi$ be a smooth, complex-valued representation of $G$ and $H$ a closed subgroup of $G$. 
\begin{itemize}
\item We say that $\pi$ is \emph{$H$-distinguished} if there exists a non-zero $H$-invariant linear form $\ell$ on the space of $\pi$, i.e., $\ell(\pi(h)v)=\ell(v)$ for all $h\in H$ and $v$ in the space of $\pi$. We denote by $\Hom_H(\pi,1)$ the space of $H$-invariant linear forms on $\pi$. 
\item More generally, for a character $\chi$ of $H$ we say that $\pi$ is $(H,\chi)$-distinguished if the space $\Hom_H(\pi,\chi)$ of $H$-equivariant linear forms on $\pi$ is non-zero.
\end{itemize}
\end{definition}
By Frobenius reciprocity we have a natural linear isomorphism
\begin{equation}\label{eq: frob rec}
\Hom_H(\pi,\chi \delta_H^{1/2}\delta_G^{-1/2})\simeq\Hom_G(\pi,\Ind_H^G(\chi)).
\end{equation}

\subsubsection{} We state the following simple observation.

\begin{lemma}\label{drmk: ist quot}
Let $\pi$ and $\sigma$ be smooth, complex-valued representations of $G$ so that $\sigma$ is a quotient of $\pi$, $H$ is a closed subgroup of $G$ and $\chi$ is a character of $H$. If $\sigma$ is $(H,\chi)$-distinguished then $\pi$ is $(H,\chi)$-distinguished.
\end{lemma}
\begin{proof}
Note that composition with the projection $\pi\rightarrow \sigma$ defines an imbedding $\Hom_H(\sigma, \chi)\hookrightarrow \Hom_H(\pi,\chi)$. The lemma follows.
\end{proof}

The lemma allows us to reduce some distinction questions to induced representations (e.g. using the Langlands classification). Its converse need not be true. 

\subsubsection{}
We record here another simple observation related to the converse problem, distinction of subquotients of a distinguished representation.
\begin{lemma}\label{lem: dist comp}
Let $\pi$ be a smooth complex-valued representation of $G$, $H$ a closed subgroup of $G$ and $\chi$ a character of $H$.
Let $0=\pi_0\subseteq \pi_1\subseteq\cdots\subseteq \pi_k=\pi$ be a filtration of $\pi$ by sub-representations.
If $\pi$ is $(H,\chi)$-distinguished, then there exists $i\in\{1,\dots,k\}$ such that $\pi_i/\pi_{i-1}$ is $(H,\chi)$-distinguished.

In particular, if $\pi$ is of finite length and $(H,\chi)$-distinguished then there exists an irreducible subquotient $\sigma$ of $\pi$ that is $(H,\chi)$-distinguished. 
\end{lemma}
\begin{proof}
If $0\ne \ell\in\Hom_H(\pi,\chi)$ then there exists $i\in\{1,\dots,k\}$ minimal such that $\ell|_{\pi_i}\ne 0$. Thus, $\ell$ defines a non-zero element of $\Hom_H(\pi_i/\pi_{i+1},\chi)$. Since a finite length representation has such a finite filtration with irreducible quotients the rest of the lemma follows.
\end{proof}


\subsubsection{}
Let $\Alg(G)$ be the category of complex valued, smooth, admissible representations of $G$ of finite length and $\Irr(G)$ the class of irreducible representations in $\Alg(G)$. 

Let $\pi^\vee$ denote the contragredient of a representation $\pi\in \Alg(G)$. Then $(\pi^\vee)^\vee\simeq\pi$ and $\pi\in \Irr(G)$ if and only if $\pi^\vee\in \Irr(G)$.

\subsection{Notation for $\GL_n(F)$}

Let $F$ be a non-archimedean local field of characteristic different than two. For $n\in \N$, let $G_n=\GL_n(F)$. By convention, let $G_0$ be the trivial group. 

\subsubsection{}
Fix $n$ and let $G=G_n$. 
Let $B=T \ltimes N$ be the standard Borel subgroup of $G$ consisting of uppertriangular matrices with its standard Levi decomposition. Here $T$ is the subgroup of diagonal matrices and $N=N_n$ is the unipotent radical of $B$.

\subsubsection{}
A parabolic subgroup of $G$ that contains $B$ is called standard. Standard parabolic subgroups of $G$ are in bijection with decompositions of $n$.
For a decomposition $\alpha=(n_1,\dots,n_k)$ of $n$ let $P_\alpha=M_\alpha\ltimes U_\alpha$ be the standard parabolic subgroup of $G$ consisting of block uppertriangular matrices with standard Levi subgroup
\[
M_\alpha=\{\diag(g_1,\dots,g_k):g_i\in G_{n_i},\,i=1,\dots,k\}\simeq G_{n_1}\times \cdots\times G_{n_k}
\]
and unipotent radical $U_\alpha$.

\subsubsection{}
The Weyl group $N_G(T)/T$ of $G$ is isomorphic to the permutation group $S_n$ of $n$ elements. We identify it with the subgroup $W=W_G$ of permutation matrices in $G$. Let $w_n=(\delta_{i,n+1-j})\in W$ be the longest Weyl element.
By the Bruhat decompositon $W$ is a complete set of representatives for the double coset space $B\bs  G/B$.

\subsubsection{}
More generally, for a decomposition $\alpha=(n_1,\dots,n_k)$ of $n$, $W_{M_\alpha}=W\cap M_\alpha\simeq S_{n_1}\times\cdots\times S_{n_k}$ is the Weyl group of $M_\alpha$.
If $P=M\ltimes U$ and $Q=L\ltimes V$ are standard parabolic subgroups of $G$ with their standard Levi decompositions then $w\mapsto PwQ$ defines a bijection 
\[
W_M\bs W/W_L\simeq P\bs G/Q.
\]
Further more, every double coset in $W_M\bs W/W_L$ contains a unique element of minimal length. Denote by ${}_MW_L$ the set of elements $w\in W$ that are of minimal length in $W_M wW_L$. Then ${}_MW_L$ is a complete set of representatives for $P\bs G/Q$.

For every $w\in {}_MW_L$ the group 
\[
P(w)=M(w)\ltimes U(w)=M\cap wQw^{-1}
\]
is a standard parabolic subgroup of $M$ with its standard Levi decomposition, where
\[
M(w)=M\cap wLw^{-1}\ \ \ \text{and}\ \ \ U(w)=M\cap wVw^{-1}.
\]


\subsection{Representations of $\GL_n(F)$}

We recall some well known facts and set the notation for representations of $G_n$.
Let $\Alg$ be the disjoint union of $\Alg(G_n)$ for all $n\in \Z_{\ge 0}$. Let $\Irr$ be the subset of irreducible representations in $\Alg$ and $\Cusp$ be the subset of cuspidal representations in $\Irr$.

\subsubsection{Parabolic Induction}

Set $G=G_n$. Let $P=M\ltimes U$ and $Q=L\ltimes V$ be standard parabolic subgroups of $G$ with their standard Levi decompositions.
Assume further that $Q$ is a subgroup of $P$. The functor $\ip_{M,L}:\Alg(L)\rightarrow\Alg(M)$ of normalized parabolic induction is defined as follows. As noted above $M\cap Q=L\ltimes (M\cap V)$ is a standard parabolic subgroup of $M$. For $\rho\in \Alg(L)$ we consider $\rho$ as a representation of $M\cap Q$ trivial on its unipotent radical $M\cap V$ and set
\[
\ip_{M,L}(\rho)=\ind_{M\cap Q}^M(\rho).
\]
The functor $\ip_{M,L}$ is exact and we have
\[
\ip_{M,L}(\rho)^\vee\simeq \ip_{M,L}(\rho^\vee).
\]

Let $\alpha=(n_{1},\dots,n_{k})$ be a decomposition of $n$. Assume that $M=M_\alpha$ and let $\rho_i\in\Alg(G_{n_i})$, $i=1,\dots,k$. Then $\rho=\rho_1\otimes \cdots \otimes \rho_k\in \Alg(M)$. 
Set
\[
\rho_1\times\cdots \times \rho_k=\ip_{G,M}(\rho).
\]

\subsubsection{Jacquet module}
The functor $\ip_{M,L}$ admits a left adjoint, namely, the normalized Jacquet functor $\jm_{L,M}:\Alg(M)\rightarrow\Alg(L)$. For $\sigma\in \Alg(M)$, $\jm_{L,M}(\sigma)$ is the representation of $L$ on the space of $V\cap M$-coinvariants of $\sigma$ induced by the action $\delta_{Q\cap M}^{-1/2}\sigma$. It is also an exact functor and for $\sigma\in \Alg(M)$ and $\rho\in \Alg(L)$ we have the natural linear isomorphism (Frobenius reciprocity):
\begin{equation}\label{eq: 1st adj}
\Hom_M(\sigma,\ip_{M,L}(\rho))\simeq \Hom_L(\jm_{L,M}(\sigma),\rho).
\end{equation}

Let $\beta_i$ be the decomposition of $n_i$, $i=1,\dots,k$ so that $L=M_{(\beta_1,\dots,\beta_k)}$. For representations $\pi_i\in \Alg(G_{n_i})$, $i=1,\dots,k$ we have
\begin{equation}\label{eq: trans jm}
\jm_{L,M}(\pi_1\otimes\cdots\otimes\pi_k)=\jm_{M_{\beta_1}, G_{n_1}}(\pi_1)\otimes\cdots\otimes \jm_{M_{\beta_k}, G_{n_k}}(\pi_k).
\end{equation}

\subsubsection{The cuspidal support}\label{def: cusp supp}

For every $\pi\in \Irr$ there exist $\rho_1,\dots,\rho_k\in \Cusp$, unique up to rearrangement, so that $\pi$ is isomorphic to a subrepresentation of $\rho_1\times \cdots \times \rho_k$. Let $\supp(\pi)=\{\rho_i:i=1,\dots,k\}$ be the support of $\pi$.\footnote{The support is often considered as a multi-set. Only the underlying set is relevant to us.} 

For $\sigma_1,\dots,\sigma_k\in \Irr$ let $\supp(\sigma_1\otimes\cdots\otimes\sigma_k)=\cup_{i=1}^k\supp(\sigma_i)$.
For any standard Levi subgroup $M$ of $G$ and $\pi\in \Alg(M)$ let $\{\pi_1,\dots,\pi_t\}$ be the set of irreducible components (subquotients) of $\pi$ and set
\[
\supp(\pi)=\cup_{i=1}^t \supp(\pi_i).
\]
As a simple consequence of the geometric lemma of Bernstein and Zelevinsky \cite[\S 2.12]{MR0579172} and exactness we have
\begin{equation}\label{eq: supp jm}
\supp(\jm_{M,G}(\pi))\subseteq\supp(\pi), \ \pi\in \Alg(G).
\end{equation}

\subsubsection{Generic representations}\label{sss: gen}
Let $\psi$ be a non-trivial character of $F$.
We further denote by $\psi=\psi_n$ the character of $N_n$ defined by
\[
\psi(u)=\psi(\sum_{i=1}^{n-1} u_{i,i+1}), \ u=(u_{i,j})\in N_n. 
\]
\begin{definition}
A representation $\pi\in\Alg(G_n)$ is called \emph{generic} if it is $(N_n,\psi)$-distinguished.
\end{definition}



\section{Non-degenerate skew-symmetric matrices and parabolic orbits}\label{s: sym sp}

We recall here the analysis of double cosets and related data that are relevant to the study of induced representations of $G_{2n}$ that are distinguished by the symplectic group. 
\subsection{The symmetric space}

Fix $n\in N$ and let $G=G_{2n}$. 

\subsubsection{}\label{sss: sp not}
Let 
\[
H=H_n=\Sp_{2n}(F)=\{g\in G:{}^t g J g=J\}
\]
where
\[
J=J_n=\begin{pmatrix} & w_n \\ -w_n & \end{pmatrix}.
\]
Note that $H=G^\inv$ is the group of fixed points in $G$ of the involution $\inv$ defined by
\[
\inv(g)=J\,{}^t g^{-1} J^{-1}.
\]

\subsubsection{}
For $\pi\in\Alg(G_n)$, since $n$ will not always be specified,  we adopt throughout the following convention. We say that $\pi$ is $\Sp$-distinguished if $n$ is even and $\pi$ is $\Sp_n(F)$-distinguished. If in addition $\pi\in \Irr$ we say that $\pi$ admits a symplectic model (see \eqref{eq: frob rec}). 
Recall the following simple observation.

By a result of Gelfand and Kazhdan, \cite{MR0404534}, we have $\pi^\inv\simeq \pi^\vee$ for every $\pi\in \Irr$. We therefore have
\begin{lemma}\label{lem: cont dist}
A representation $\pi\in\Irr$ is $\Sp$-distinguished if and only if $\pi^\vee$ is $\Sp$-distinguished. \qed
\end{lemma}

\subsubsection{}
Consider the symmetric space
\[
X=\{x\in G: x\inv(x)=I_{2n}\}
\]
with the $G$-action
\[
g\cdot x=gx\inv(g)^{-1}.
\]
Note that $XJ$ is the space of skew-symmetric matrices in $G$ and 
\[
(g\cdot x)J=g(xJ){}^tg.
\]
Therefore $X$ is a homogeneous $G$-space.
The map $(g\mapsto g\cdot I_{2n}): G\rightarrow X$ defines an isomorphism $G/H\simeq X$ of $G$-spaces.

\subsubsection{}
For a subgroup $Q$ of $G$ and a $Q$-invariant subspace $Y$ of $X$ we denote by $Q\bs Y$ the set of $Q$-orbits in $Y$. For $x\in X$ let $Q_x=\{g\in Q:g\cdot x=x\}$ be the stabilizer in $Q$ of $x$.

Applications of the geometric lemma to the study of $H$-distinguished induced representations of $G$ require the study of orbits in $X$ by the group from which we induce. Since parabolic induction is central to the classification of $\Irr$ we recall next the study of orbits in $X$ under a standard parabolic subgroup $P=M\ltimes U$. 

Of particular interest for these applications are choices of orbit representatives $x$ for which we can provide explicit description of the stabilizer $M_x$ and the restriction to $M_x$ of $\delta_P^{1/2}\delta_{P_x}^{-1}$.

We refer to \cite[\S3]{MR2254544} and \cite[\S3.1]{MR2248833} for proofs of the results presented in Sections \ref{X/B} and \ref{X/P}.
\subsection{Borel orbits in $X$}\label{X/B}

We begin with the Borel orbits.
\subsubsection{}
Note that both $B$ and $T$ are $\inv$ stable. 
In particular $\inv$ defines an involution on $W$ that we continue to denote by $\inv$. We have 
\[
\inv(w)=w_{2n}ww_{2n}^{-1},\,w\in W.
\]
 It also follows that the map $( B\cdot x\mapsto BxB ): B\bs X \rightarrow B\bs G/B$ is well defined. By the Bruhat decomposition this defines a map from $B\bs X$ to $W$. Since $\inv(BxB)=(BxB)^{-1}$ it follows that every $w$ in the image of this map satisfies $w\inv(w)=e$ (the identity element of $W$). We refer to such permutations as twisted involutions.

\subsubsection{}
Let 
\[
[w_{2n}]=\{ww_{2n}w^{-1}:w\in W\} 
\]
be the $W$-conjugacy class of the longest Weyl element. It is the set of involutions without fixed points in $W$.

Note that the set of twisted involutions in $W$ is precisely $[w_{2n}]w_{2n}$. In fact we have
\begin{lemma}
The map $( B\cdot x\mapsto B\cdot x\cap T ):B\bs X\rightarrow T\bs (X\cap N_G(T))$ and the natural map from $T\bs (X\cap N_G(T))$ to $[w_{2n}]w_{2n}$ are both bijective. \qed
\end{lemma}
\subsubsection{}
Fix a Borel orbit $B\cdot x\in B\bs X$. We may and do assume that $x\in X\cap N_G(T)$ and let $w\in [w_{2n}]$ be such that $x\in T ww_{2n}$.
Below is a description of $T_x$ and of the restriction to $T_x$ of $\delta_B^{1/2}\delta_{B_x}^{-1}$.
Note first that
\[
T_{I_{2n}}=T\cap H=\{\diag(a_1,\dots,a_n,a_n^{-1},\dots,a_1^{-1}): a_i\in F^*,\,i=1,\dots,n\}
\]
and
\[
(\delta_B^{-1/2}\delta_{B_{I_{2n}}})(\diag(a_1,\dots,a_n,a_n^{-1},\dots,a_1^{-1}))=\abs{a_1}_F\cdots\abs{a_n}_F.
\]
We summarize the relevant results.
\begin{lemma}
For every orbit in $B\bs X$ there exists a unique $\tau\in [w_{2n}]$ and a representative $x\in N_G(T)$ such that
\begin{enumerate}
\item $T_x=\{\diag(a_1,\dots,a_{2n})\in T: a_{\tau(i)}=a_i^{-1},\,i=1,\dots,n\}$;
\item $(\delta_B^{-1/2}\delta_{B_x})(t)=\prod_{i<\tau(i)}\abs{a_i}_F$ for every $t=\diag(a_1,\dots,a_{2n})\in T_x$. \qed
\end{enumerate}
\end{lemma}

\subsection{$P$-orbits in $X$}\label{X/P}

Let $\alpha=(n_1,\dots,n_k)$ be a decomposition of $2n$ and let $P=P_\alpha=M\ltimes U$. Note that $\inv(P)=P_{(n_k,\dots,n_1)}$ and $\inv(M)=w_{2n}Mw_{2n}^{-1}=M_{(n_k,\dots,n_1)}$.

\subsubsection{} The $P$-orbits in $X$ are in bijection with certain twisted involutions.
\begin{lemma}\label{lem: P-bij}
The map $(P\cdot x\mapsto Px\inv(P) ) :P\bs X\rightarrow P\bs G/\inv(P)\simeq {}_MW_{\inv(M)}$ defines a bijection 
\[
P\bs X\simeq {}_MW_{\inv(M)}\cap [w_{2n}]w_{2n}. 
\] \qed
\end{lemma}
\subsubsection{} 
Let 
\[
\ob_M: {}_MW_{\inv(M)}\cap [w_{2n}]w_{2n}\rightarrow P\bs X
\]
denote the bijection of Lemma \ref{lem: P-bij}.
Recall that for $w\in {}_MW_{\inv(M)}$  
\[
M(w)=M\cap w\inv(M)w^{-1}
\]
is a standard parabolic subgroup of $M$.
\begin{lemma}\label{lem good reps}
For every $w\in {}_MW_{\inv(M)}\cap [w_{2n}]w_{2n}$ we have that $\ob_M(w)\cap M(w)w$ is a single $M(w)$-orbit. In particular, it is not empty. \qed
\end{lemma}

\subsubsection{Admissible orbits}
\begin{definition} 
We say that $w\in {}_MW_{\inv(M)}\cap [w_{2n}]w_{2n}$ (or the corresponding $P$-orbit $\ob_M(w)$) is $M$-admissible if $M(w)=M$, i.e., if $ww_{2n}\in N_G(M)$. 
\end{definition}
Thus, $w\in {}_MW_{\inv(M)}\cap [w_{2n}]w_{2n}$ is $M$-admissible if and only if the intersection $\ob_M(w)\cap N_G(M)w_{2n}$ is not empty.
In particular $\ob_M$ restricts to a bijection 
\[
 ({}_MW_{\inv(M)}\cap [w_{2n}]w_{2n}\cap N_G(M)w_{2n})\ \ \simeq \ \ (M-\text{admissible orbits in}\ X).
\]

The $P$-orbits in $X$ are studied in terms of certain $L$-admissible orbits for Levi subgroups $L$ of $M$. More precisely, to the $P$-orbit $\ob_M(w)$ we associate a certain $M(w)$-admissible orbit. 
We therefore begin by describing the relevant data for $M$-admissible orbits.

Let 
\[
S_2[\alpha]=\{\tau\in S_k: \tau^2=e,\,n_{\tau(i)}=n_i,\,i=1,\dots,k\  \text{and}\ n_i\ \text{is even if} \ \tau(i)=i\}.
\]
The admissible $M$-orbits are in bijection with $S_2[\alpha]$. Before we state the general results we provide examples of prototypes of admissible orbits.
\subsubsection{} 
Assume that $k=s+2t$, $n_i=n_{k+1-i}$, $i=1,\dots,t$ and $n_i$ is even for $i=t+1,\dots,t+s$, i.e., $\alpha$ is of the form
\[
\alpha=(n_1,\dots,n_t, 2m_1,\dots,2m_s,n_t,\dots,n_1).
\]
Let
\[
x=\diag(I_N,J_{(m_1,\dots,m_s)}J_m^{-1},I_N)=\begin{pmatrix} & & w_N \\ & J_{(m_1,\dots,m_s)} & \\ -w_N & & \end{pmatrix}J_n^{-1}\in X
\]
where
\[
J_{(m_1,\dots,m_s)}=\diag(J_{m_1},\dots,J_{m_s}),\ \ \ N=n_1+\cdots+n_t \ \ \ \text{and}\ \ \ m=m_1+\cdots+m_s.
\]
Note that $xJ_n$ is a skew-symmetric matrix in $N_G(M)$ and therefore $P\cdot x$ is $M$-admissible. 

For every $d\in \N$ consider the involution $g\mapsto g^*$ on $G_d$ defined by $g^*=w_d\,{}^tg^{-1}w_d^{-1}$. We have
\[
M_x=\{\diag(g_1,\dots,g_t,h_1,\dots,h_s,g_t^*,\dots,g_1^*):g_1,\dots,g_t\in G_{n_i},\,h_1,\dots,h_s\in H_{m_j}\}
\]
and
\[
(\delta_P^{-1/2}\delta_{P_x})(\diag(g_1,\dots,g_t,h_1,\dots,h_s,g_t^*,\dots,g_1^*))=\prod_{i=1}^t \abs{\det g_i}_F.
\]
\subsubsection{} 
We now return to the general setting where $\alpha$ is any decomposition of $2n$. 
\begin{lemma}\label{admP-orb}
There is a bijection between the $M$-admissible $P$-orbits in $X$ and $S_2[\alpha]$ that satisfies the following properties. Let $w\in   {}_MW_{\inv(M)}\cap [w_{2n}]w_{2n}\cap N_G(M)w_{2n}$ and let $\tau\in S_2[\alpha]$ correspond to $\ob_M(w)$. Then, there exists $x=x_{M,w}\in X\cap Mw$ such that:
\begin{enumerate}
\item $M_x=\{\diag(g_1,\dots,g_k):g_{\tau(i)}=g_i^* \ \text{if}\ \tau(i)\ne i\ \text{and} \ g_i\in H_{n_i/2}\ \text{if}\ \tau(i)=i\}$
\item $(\delta_P^{-1/2}\delta_{P_x})(\diag(g_1,\dots,g_k))=\prod_{i<\tau(i)} \abs{\det g_i}_F$. \qed
\end{enumerate}
\end{lemma}

\subsubsection{} 
Every $\tau\in S_k$ defines a unique $w_\tau\in W$ such that for every $g=\diag(g_1,\dots,g_k)\in M$ we have
\[
w_\tau gw_\tau^{-1}=\diag(g_{\tau^{-1}(1)},\dots,g_{\tau^{-1}(k)}).
\]

\begin{remark}
In fact, the relation between $w$ and $\tau$ in the above lemma is characterized by $\diag(w_{n_1},\dots,w_{n_k}) w_\tau w_{2n}= w$. 
\end{remark}
\subsubsection{General orbits}\label{General orbits}
Fix $w\in {}_MW_{\inv(M)}\cap [w_{2n}]w_{2n}$ and let $L=M(w)$ be the standard Levi subgroup of $M$ we associated with $w$. 
Let $\beta=(\beta_1,\dots,\beta_k)$ be such that $L=M_\beta$ where $\beta_i=(m_{i,1},\dots, m_{i,k_i})$ is a decomposition of $n_i$. On the set of indices 
\[
\indset=\{(i,j):i=1,\dots,k,\ j=1,\dots,k_i\}
\] 
we consider the lexicographic order $(i,j) \prec (i',j')$ if either $i<i'$ or $i=i'$ and $j<j'$. We further consider the partial order $(i,j)\ll(i',j')$ if $i<i'$.

Recall that by Lemmas \ref{lem: P-bij} and \ref{lem good reps}, $X\cap Lw$ is an $L$-orbit. Note that $w$ is $L$-admissible. Furthermore, for $x\in X\cap Lw$ we have $M_x=L_x$ and $P_x=Q_x$ where $Q=P_\beta$ is the standard parabolic subgroup of $G$ with Levi subgroup $L$. We may therefore apply Lemma \ref{admP-orb} with $M$ replaced by $L$. 

We consider $S_2[\beta]$ as a set of involutions on $\indset$, by identifying $(\indset,\prec)$ with the linearly ordered set $\{1,2,\dots,\abs{\indset}\}$. Let $\tau\in S_2[\beta]$ be the involution associated with $\ob_L(w)$.
Since $w\in {}_MW_{\inv(M)}$ there are more restrictions on $\tau$, it must satisfy 
\begin{equation}\label{eq: stronger cond}
\tau(i,j+1)\ll \tau(i,j),\ i=1,\dots,k,\,\ j=1,\dots,k_i-1.
\end{equation}
This implies in particular that for every $i$ there is at most one $j$ such that $\tau(i,j)=(i,j)$. 

\section{The geometric lemma}\label{sec: gl}
We recall here a special case of the geometric lemma \cite[Theorem 5.2]{MR0579172} (see also \cite[Proposition 1.17]{MR2401221}). 

As in the previous section, fix $n\in \N$ and let $G=G_{2n}$ and $H=H_n$.
Let $\alpha=(n_1,\dots,n_k)$ be a decomposition of $2n$ and let $P=P_\alpha=M\ltimes U$.
Consider the functor $\Res_H\circ \ip_{G,M}$ from $\Alg(M)$ to the category of smooth representations of $H$ where $\Res_H$ stands for restriction to $H$.
\subsection{The $H$-filtration}

For every $\sigma\in \Alg(M)$ we recall here the existence of an $H$-filtration on $\Res_H\circ \ip_{G,M}(\sigma)$ parameterized by $P\bs X$ and explicate the factors of the filtration.

\subsubsection{} By \cite[\S1.5]{MR0425030} (see also \cite[Lemma 3.1]{MR2401221}) there is a linear ordering ${}_MW_{\inv(M)}\cap [w_{2n}]w_{2n}=\{w_1,\dots,w_m\}$ so that
\[
X_i=\cup_{j=1}^i \ob_M(w_i)
\]
is open in $X$ for all $i=1,\dots,m$. 
\subsubsection{}\label{sss: clop}
The orbit of the identity $\iota_M(I_{2n})=P\cdot I_{2n}$ is closed in $X$ and we may assume that $w_m=I_{2n}$. Furthermore, if $n_i$ is even for all $i=1,\dots,k$ then the orbit $P\cdot x_M$ where
\[
x_M=J_{(n_1/2,\dots,n_k/2)}J_n^{-1} 
\]
is open in $X$ and we may assume that $\ob_M(w_1)=P\cdot x_M$. Furthermore, in this case, $w_1$ is $M$-admissible.
\subsubsection{} 
Let $\sigma\in \Alg(M)$ and let $\Nu$ be the representation space of $\ip_{G,M}(\sigma)$. Set
\[
\Nu_i=\{\varphi\in\Nu:\supp(\varphi)\subseteq X_i\},\ i=1,\dots,m
\]
then $\Nu_0:=0\subseteq \Nu_1 \subseteq \cdots\subseteq\Nu_m=\Nu$ is a filtration of $\Res_H(\ip_{G,M}(\sigma))$.
For every $i$ choose (by Lemma \ref{lem good reps}) $x_i\in \ob_M(w_i)\cap M(w_i)w_i$ and $\eta_i\in G$ such that $\eta_i\cdot I_{2n}=x_i$.

For a subgroup $A$ of a group $B$, $b\in B$ and a representation $\rho$ of $A$ we denote by $\rho^b$ the representation of $b^{-1}Ab$ on the space of $\rho$ defined by $\rho^b(b^{-1}ab)=\rho(a)$. By \cite[Proposition 3]{MR2248833} we have
\begin{lemma}\label{lem comp iso}
For every $i=1,\dots,m$ we have the isomorphism of representations of $H$
\[
\Nu_i/\Nu_{i-1}\simeq \ind_{H\cap \eta_i^{-1}P_{x_i}\eta_i}^H(\delta_{P_{x_i}}^{-1/2}(\Res_{P_{x_i}}(\delta_P^{1/2}\sigma))^{\eta_i}).
\] 
\qed
\end{lemma}
\subsubsection{Relevant orbits}
\begin{definition}
We say that $w_i\in {}_MW_{\inv(M)}\cap [w_{2n}]w_{2n}$ (or $\ob_M(w_i)$) is relevant for $\sigma$ if 
\[
\Hom_H(\Nu_i/\Nu_{i-1},1)\ne 0.
\]
\end{definition}
\subsubsection{} The following makes this property more explicit. By \cite[Corollary 1]{MR2248833} we have
\begin{lemma}\label{lem comp frob}
Fix $i$ and let $w=w_i$, $x=x_i$, $\eta=\eta_i$, $L=M(w_i)$ and $Q$ the standard parabolic subgroup of $G$ with Levi subgroup $L$. Then
\[
\Hom_H(\ind_{H\cap \eta^{-1}P_x\eta}^H(\delta_{P_x}^{-1/2}(\Res_{P_x}(\delta_P^{1/2}\sigma))^{\eta}),1)\simeq \Hom_{L_x}(\jm_{L,M}(\sigma),\delta_Q^{-1/2}\delta_{Q_x}).
\]
\qed
\end{lemma}
\subsubsection{} Combining Lemmas \ref{lem comp iso} and \ref{lem comp frob} we have
\begin{corollary}\label{cor: rel}
Let $w\in {}_MW_{\inv(M)}\cap [w_{2n}]w_{2n}$. With the above notation $w$ is relevant for $\sigma$ if and only if
\[
\Hom_{L_x}(\jm_{L,M}(\sigma),\delta_Q^{-1/2}\delta_{Q_x})\ne 0.
\]
\qed
\end{corollary}
\subsubsection{} Finally, by choosing the orbit representative $x$ as in Lemma \ref{admP-orb} (where $M$ is replaced by $L$) we explicate the condition $\Hom_{L_x}(\rho,\delta_Q^{-1/2}\delta_{Q_x})\ne 0$ for certain pure tensor representations $\rho\in \Alg(L)$.

For a representation $\pi\in \Alg(G_r)$, let $\pi^*\in \Alg(G_r)$ be the representation on the space of $\pi$ defined by $\pi^*(g)=\pi(g^*)$. By a result of Gelfand and Kazhdan (\cite{MR0404534}) we have $\pi^*\simeq\pi^\vee$ for all $\pi\in \Irr$. In the notation of Section \ref{General orbits} let
\[
\rho=\otimes_{\imath\in (\indset,\prec)}\rho_\imath
\]
and let $\tau\in S_2[\beta]$ be the involution on $\indset$ associated to $w$ by Lemma \ref{admP-orb} applied with $L$ replacing $M$. 
Assume that $\rho_\imath\in \Irr(G_{n_\imath})$ whenever $\tau(\imath)\ne \imath$.
Then by Lemma \ref{admP-orb} we have
\begin{multline}\label{dist cond}
\Hom_{L_x}(\rho,\delta_Q^{-1/2}\delta_{Q_x})\ne 0 \ \ \ \text{if and only if for all}\  \imath\in\indset \ \text{we have} \\ \rho_\imath \simeq \nu\rho_{\tau(\imath)} \ \text{whenever} \ \imath \prec  \tau(\imath) \ \text{and}\ \rho_\imath\ \text{is} \ H_{n_\imath}-\text{distinguished if} \ \tau(\imath)=\imath.  
\end{multline}

\subsubsection{}
Combining Corollary \ref{cor: rel}, \eqref{dist cond} and Lemma \ref{lem: dist comp} we have
\begin{corollary}\label{cor: fine rel} 
Let $w\in {}_MW_{\inv(M)}\cap [w_{2n}]w_{2n}$. With the above notation, if $w$ is relevant for $\sigma$ then
there exists an irreducible component $\rho=\otimes_{\imath\in (\indset,\prec)}\rho_\imath\in \Irr(L)$ of $\jm_{L,M}(\sigma)$ such that
\[
\rho_\imath \simeq \nu\rho_{\tau(\imath)} \ \text{whenever} \ \imath\prec \tau(\imath) \ \text{and}\ \rho_\imath\ \text{is} \ H_{n_\imath}-\text{distinguished if} \ \tau(\imath)=\imath. 
\] \qed
\end{corollary}

\section{First applications of the geometric lemma to $\Sp$-distinction}
In this section we apply the contribution of the open and closed orbits of the filtration defined in \S \ref{sec: gl} in order to show that certain induced representations are $\Sp$-distinguished. We further apply \S \ref{sec: gl} to reduce the study of $\Sp$-distinction on $\Irr$ to representations supported on a single cuspidal line.

\subsection{Distinction and relevant orbits} \label{ss: oc}
\subsubsection{}
The variant of the geometric lemma discussed in \S \ref{sec: gl} is often applied to show that certain induced representations are not distinguished.
This is based on the following simple observation, which is an immediate consequence of Lemma \ref{lem: dist comp} (applied with $G=H$).
\begin{lemma}\label{lem: ex orb}
Let $M$ be a standard Levi subgroup of $G$. If $\sigma\in\Alg(M)$ is such that $\ip_{G,M}(\sigma)$ is $H$-distinguished then there exists a $P$-orbit in $X$ that is relevant to $\sigma$. \qed
\end{lemma}

The reverse implication need not be true. However, there are two cases in which the geometric lemma indicates distinction.

\subsubsection{} Assume that $n_i$ is even for all $i$. Then the open $P$-orbit in $X$ is $\ob_M(w_1)=P\cdot x_M$  (see \S \ref{sss: clop}) and
\[
M_{x_{M}}=\{\diag(h_1,\dots,h_k): h_i\in H_{n_i/2},\,i=1,\dots,k\}.
\]
Let $\sigma_i\in \Alg(G_{n_i})$ be $H_{n_i/2}$-distinguished and $0\ne \ell_i\in \Hom_{H_{n_i/2}}(\sigma_i,1)$ for all $i=1,\dots,k$. Let $\sigma=\sigma_1\otimes\cdots\otimes\sigma_k$ and $\ell=\ell_1\otimes\cdots\otimes\ell_k\in \Hom_{M_{x_M}}(\sigma,1)$. The integral  
\begin{equation}\label{open int}
\tilde\ell(\varphi)=\int_{(H\cap \eta_M^{-1}M_{x_M}\eta_M)\bs H}\varphi(\eta_Mh)\ dh
\end{equation}
where $\eta_M\in G$ is such that $\eta_M\cdot I_{2n}=x_M$, defines a non-zero linear form $\tilde\ell\in \Hom_H(\Nu_1,1)$. It does not necessarily extend to an $H$-invariant linear form on $\ip_{G,M}(\sigma)$, but it lies in a holomorphic family of linear forms that do extend meromorphically. 

Note that $G_{x_M}=G^{\inv_{x_M}}$ is the fixed point group of the involution $\inv_{x_M}(g)=x_M\inv(g)x_M^{-1}$ and that $M=P\cap \inv_{x_M}(P)$. 
For $\lambda=(\lambda_1,\dots,\lambda_k)\in \C^k$ let $\sigma[\lambda]$ be the representation on the space of $\sigma$ defined by 
\[
\sigma[\lambda](\diag(g_1,\dots,g_k))=\abs{\det g_1}_F^{\lambda_1}\cdots\abs{\det g_k}_F^{\lambda_k}\sigma(\diag(g_1,\dots,g_k)).
\]
The representations $\ip_{G,M}(\sigma[\lambda])$ can all be realized in the same space $\Nu$ and then the $H$-filtration $\{\Nu_i\}_{i=1}^n$ is independent of $\lambda$. The following follows from \cite[Theorem 2.8]{MR2401221}.
\begin{lemma}\label{BD her}
With the above notation and assumptions, there is a non-zero meromorphic function $(\lambda\mapsto \ell_\lambda):\C^k\rightarrow \Nu^*$ that satisfies $\ell_\lambda\in \Hom_H(\ip_{G,M}(\sigma[\lambda]),1)$ whenever holomorphic at $\lambda$.
\end{lemma}
\subsubsection{Hereditary property of $\Sp$-distinction}
This implies the hereditary property of $\Sp$-distinction.
\begin{corollary}\label{open dist}
Assume that $n_i$ is even for all $i=1,\dots,k$. 
Let $\sigma_i\in \Alg(G_{n_i})$ be $H_{n_i/2}$-distinguished for all $i=1,\dots,k$ Then $\sigma_1\times\cdots\times\sigma_k$ is $H$-distinguished.
\end{corollary}
\begin{proof}
This is immediate from Lemma \ref{BD her} by taking a leading term at $\lambda=0$ of $\ell_\lambda$ at a complex line through zero in a generic direction. 
\end{proof}

\subsubsection{Distinction by the closed orbit}
When a closed orbit is relevant, the geometric lemma directly implies distinction.
\begin{lemma}\label{lem: closed}
Let $\sigma_1,\dots,\sigma_t \in \Irr$, $\rho_1,\dots,\rho_s\in \Alg$ and assume that $\rho_i$ is $\Sp$-distinguished for $i=1,\dots,s$ (allow the case $s=0$). Then 
\[
\nu\sigma_1\times\cdots\times \nu\sigma_k\times \rho_1\times \cdots\times \rho_s\times \sigma_k\times\cdots\times\sigma_1
\]
is $\Sp$-distinguished.
\end{lemma}
\begin{proof}
It follows from Corollary \ref{open dist} that $\rho= \rho_1\times \cdots\times \rho_s$ is $\Sp$-distinguished. Let $G=G_{2n}$ and $M$ its standard Levi subgroup so that
\[
\sigma=\nu\sigma_1\otimes\cdots\otimes \nu\sigma_k\otimes \rho\otimes \sigma_k\times\cdots \otimes\sigma_1
\]
is a representation of $M$. Then $\ip_{G,M}(\sigma)\in \Alg(G)$. Let $P$ be the standard parabolic subgroup with Levi subgroup $M$. Then the closed orbit $P\cdot I_{2n}$ is relevant to $\sigma$ by \eqref{dist cond}. 
By \S \ref{sss: clop}, $\Hom_H(\Nu/\Nu_{m-1},1)\ne 0$ and therefore by Lemma \ref{drmk: ist quot} (applied with $G=H$), $\ip_{G,M}(\sigma)$ is $\Sp$-distinguished.

\end{proof}

\subsection{Reduction to cuspidal $\Z$-lines} \label{ss: cusp lines}

Let $\nu=\abs{\det}_F$ on $G_n$ for any $n\in \N$. 
\subsubsection{}For $\rho\in \cusp$ let $\rho^\Z=\{\nu^n\rho:n\in\Z\}$ be the $\Z$-line through $\rho$. Denote by $<$ the order on $\rho^\Z$ induced by the standard order on $\Z$ (so that $\rho< \nu\rho$).
\begin{definition}
A representation $\pi\in \Alg$ is called \emph{rigid} if $\supp(\pi)\subseteq \rho^\Z$ for some $\rho\in\Cusp$.
\end{definition}

\subsubsection{}
Every element of $\Irr$ has a unique decomposition as a product of rigid representations supported on disjoint cuspidal lines.
Indeed, by \cite[Proposition 8.6]{MR584084} we have
\begin{lemma}\label{lem; disj is irr}
For every $\pi\in \Irr$ there exist $\rho_1,\dots,\rho_k\in \Cusp$, so that $\rho_i^\Z\cap \rho_j^\Z=\emptyset$ for all $i\ne j$, and $\pi_1,\dots,\pi_k\in\Irr$ so that $\supp(\pi_i)\subseteq \rho_i^\Z$ and $\pi=\pi_1\times\cdots\times\pi_k$.
\qed
\end{lemma}

\subsubsection{}
Another application of the geometric lemma will allow us to reduce the question of $\Sp$-distinction of irreducible representations to those supported on a single cuspidal line. Indeed, if $\pi=\pi_1\times\cdots\times\pi_k$ is a decomposition as in Lemma \ref{lem; disj is irr} then 
\begin{equation}\label{eq: tensor H}
\Hom_{\Sp}(\pi,1)\simeq\Hom_{\Sp}(\pi_1,1)\otimes \cdots\otimes \Hom_{\Sp}(\pi_k,1).
\end{equation}
Here, we write $\Hom_{\Sp}(\pi,1)=\Hom_{\Sp_{2n}(F)}(\pi,1)$ for any $\pi\in \Alg(G_{2n})$.

In fact, we prove \eqref{eq: tensor H} for a slightly more general setting for which we need to introduce some more terminology.

\subsubsection{}
Consider the graph $\mathcal{E}$ with $\Cusp$ as the set of vertices and an edge between $\rho$ and $\nu\rho$ for every $\rho\in \Cusp$. 

For every finite subset $V\subseteq \Cusp$ let $\mathcal{E}_V$ be the induced graph on the set of vertices $V$ and $\mult_V$ the set of connected components of $\mathcal{E}_V$. Every connected component $\Delta\in\mult_V$ is of the form $\Delta=\{\nu^i\rho:i=a,\dots,b\}$ for some $\rho\in \Cusp$ and integers $a\le b$.
\begin{definition}\label{def: tot disj}
We say that finite subsets $V,\,V'\subseteq\Cusp$ are \emph{totally disjoint} if either $V$ and $V'$ are contained in disjoint cuspidal $\Z$-lines or they satisfy the following property. For every $\Delta\in \mult_V$ and $\Delta'\in \mult_{V'}$ we have that either $\nu\rho< \rho'$ for all $\rho\in\Delta$ and $\rho'\in \Delta'$ or $\nu\rho'< \rho$ for all $\rho\in\Delta$ and $\rho'\in \Delta'$. (Equivalently, $\Delta\cup \Delta'$ is not connected in $\mathcal{E}_{V\cup V'}$.)
\end{definition}
As a consequence of \cite[Proposition 8.5]{MR584084} we have
\begin{lemma}\label{lem; tot disj is irr}
If $\pi_1,\dots,\pi_k\in \Irr$ are such that $\supp(\pi_i)$ and $\supp(\pi_j)$ are totally disjoint for all $i\ne j$ then $\pi_1\times \cdots\times \pi_k\in \Irr$.
\qed
\end{lemma}

\subsubsection{} We now show that \eqref{eq: tensor H} holds for totally disjoint decompositions as in Lemma \ref{lem; tot disj is irr}.
The following is a small generalization of \cite[Lemma 3.4]{MR3227442}.
\begin{lemma}\label{lem: useful cusp}
Let $\pi_1,\dots,\pi_k\in \Alg$ be such that $\supp(\pi_i)$ and $\supp(\pi_j)$ are totally disjoint for all $i\ne j$. Then $\pi=\pi_1\times\cdots\times\pi_k$ is $\Sp$-distinguished if and only if $\pi_i$ is $\Sp$-distinguished for all $i=1,\dots,k$. In particular, if $\pi\in \Irr$ then \eqref{eq: tensor H} holds.
\end{lemma}
\begin{proof}
The `only if' part follows from Corollary \ref{open dist}. We prove the `if' part.
Let $\sigma=\pi_1\otimes\cdots\otimes \pi_k$, $n$ be such that $\pi_1\times\cdots\times\pi_k\in\Alg(G_n)$ and $\alpha$ the decomposition of $n$ such that $\sigma\in \Alg(M_\alpha)$. Set $G=G_n$ and $M=M_\alpha$. Assume that $\pi_1\times\cdots\times\pi_k$ is $\Sp$-distinguished. 

By Lemma \ref{lem: ex orb},
$\sigma$ admits a relevant orbit, let $w\in {}_MW_{\inv(M)}\cap [w_{2n}]w_{2n}$ be relevant to $\sigma$. 
Apply the notation of \S \ref{General orbits}.
By Corollary \ref{cor: fine rel}, there exists an irreducible component $\rho$ of $\jm_{L,M}(\sigma)$ that satisfies \eqref{dist cond}.
Then $\rho=\otimes_{\imath\in (\indset,\prec)} \rho_\imath$ where $\rho_{i,1}\otimes\cdots\otimes\rho_{i,k_i}$ is an irreducible component of $\jm_{M_{\beta_i},G_{n_i}}(\pi_i)$ for all $i=1,\dots,k$ (see \eqref{eq: trans jm}). In particular, $\supp(\rho_{i,j})\subseteq \supp(\jm_{M_{\beta_i},G_{n_i}}(\pi_i))\subseteq \supp(\pi_i)$ for all $i$ (see \eqref{eq: supp jm}).

Assume that there exists $\imath\in\indset$ such that $\imath\prec\tau(\imath) $ and let $\imath=(i,j)$ and $\tau(\imath)=(i',j')$. Then $\rho_\imath\simeq \nu\rho_{\tau(\imath)}$ and by \eqref{eq: stronger cond}, $i\ne i'$. In particular, there exists $\rho'\in \supp(\rho_{\tau(\imath)})\subseteq \supp(\pi_{i'})$ such that $\nu\rho'\in\supp(\rho_\imath)\subseteq \supp(\pi_i)$. This contradicts the total disjointness of $\supp(\pi_i)$ and $\supp(\pi_{i'})$. Therefore, $\tau$ is the trivial involution. Now \eqref{eq: stronger cond} implies that $w$ is $M$-admissible and $\pi_i$ is $\Sp$-distinguished for all $i=1,\dots,k$ as required.

The isomorphism \eqref{eq: tensor H} now follows from \cite[Theorem 2.4.2]{MR1078382} (local multiplicity one for symplectic models).
\end{proof}

\section{Representations of $\GL_n(F)$}\label{s: LZ cl}
Before we continue with further applications of the geometric lemma to $\Sp$-distinction, we need to introduce the segment notation of Zelevinsky, and the Langlands and Zelevinsky classifications of $\Irr$. We refer to \cite{MR584084} for the results stated in this section.

\subsection{Segment Representations}\label{info: segments}
By a segment of cuspidal representations we mean a set
\[
[a,b]_{(\rho)}=\{\nu^i\rho:i=a,\,a+1,\dots,b\}
\]
where $\rho\in \Cusp$ and $a \le b$ are integers. By convention, the empty set is also considered a segment.
\subsubsection{} For a segment $\Delta=[a,b]_{(\rho)}$ as above, the representation $\nu^a\rho\times \nu^{a+1}\rho \times\cdots\times\nu^b\rho$ has a unique irreducible subrepresentation that we denote by $Z(\Delta)$ and a unique irreducible quotient that we denote by $L(\Delta)$.

\subsubsection{} We remark that $\Delta\mapsto L(\Delta)$ is a bijection between the set of segments of cuspidal representations and the subset of essentially square-integrable representations in $\Irr$.

\subsubsection{} Also, $Z(\Delta)$ is the unique irreducible quotient and $L(\Delta)$ is the unique irreducible subrepresentation of $\nu^b\rho\times\cdots\times\nu^{a+1}\rho\times\nu^a\rho$. By convention, if the segment $\Delta$ is empty, then both $L(\Delta)$ and $Z(\Delta)$ are taken to be the trivial representation of the trivial group.

\subsubsection{} We denote by $b(\Delta)=\nu^a\rho$ the beginning, $e(\Delta)=\nu^b\rho$ the end and $\length(\Delta)=\abs{\Delta}=b-a+1$ the length of $\Delta$.
Let $\nu\Delta=[a+1,b+1]_{(\rho)}$.

\subsubsection{}
Let $\Delta$ and $\Delta'$ be segments of cuspidal representations.
We say that $\Delta$ \emph{precedes} $\Delta'$ and write $\Delta\prec\Delta'$ if both $\Delta$ and $\Delta'$ are contained in some $\Z$-line $\rho^\Z\subseteq \Cusp$, $b(\Delta)<b(\Delta')$, $e(\Delta)<e(\Delta')$ and $e(\Delta) \ge b(\Delta')-1$. 

The segments $\Delta$ and $\Delta'$ are called \emph{linked} if either $\Delta\prec\Delta'$ or $\Delta'\prec\Delta$. Equivalently, $\Delta$ and $\Delta'$ are linked if and only if neither of them is contained in the other and their union is a segment.

\subsubsection{}\label{sec: standard} For our conventions regarding multi-sets see \S \ref{sss: mult}.
Let $\OO$ be the set of multi-sets of segments of cuspidal representatons.
An order $\mult=\{\Delta_1,\dots,\Delta_t\}\in\OO$ on a multi-set $\mult$ is of \emph{standard form} if $\Delta_i\not \prec\Delta_j$ for all $i<j$. 

Every $\mult\in \OO$ admits at least one standard order. Indeed, if for example $e(\Delta_1)\ge \cdots\ge e(\Delta_t)$ then $\{\Delta_1,\dots,\Delta_t\}$ is in standard form.

\subsubsection{The Zelevinsky classification}\label{sss: Z cl} Let $\mult=\{\Delta_1,\dots,\Delta_t\}\in\OO$ be ordered in standard form. The representation
\[
\zeta(\mult)= Z(\Delta_1) \times\cdots \times Z(\Delta_t)
\]
is independent of the choice of order of standard form. 
It has a unique irreducible submodule that we denote by $Z(\mult)$. 

The Zelevinsky classification says that the map $(\mult\mapsto Z(\mult)):\OO\rightarrow \Irr$ is a bijection.

\subsubsection{}
The representation
\[
\tilde\zeta(\mult)=Z(\Delta_t) \times\cdots \times Z(\Delta_1).
\]
is also independent of the choice of standard order on $\mult$ and $Z(\mult)$ is the unique irreducible quotient of $\tilde\zeta(\mult)$.

\subsubsection{The Langlands classification}\label{sss: L cl} Let $\mult=\{\Delta_1,\dots,\Delta_t\}\in\OO$ be ordered in standard form. The representation
\[
\lambda(\mult)=L(\Delta_1)\times\cdots\times L(\Delta_t)
\]
is independent of the choice of order of standard form. 
It has a unique irreducible quotient that we denote by $L(\mult)$. 

The Langlands classification says that the map $( \mult\mapsto L(\mult)) :\OO\rightarrow \Irr$ is a bijection.
\subsubsection{The Zelevinsky involution}
It follows from \S \ref{sss: Z cl} and \S \ref{sss: L cl} that for any $\mult\in \OO$ there exists a unique $\mult^t\in \OO$ such that $Z(\mult)=L(\mult^t)$. 

The function $\mult\mapsto \mult^t$ is an involution on $\OO$.
For $\pi=Z(\mult)\in \Irr$ let $\pi^t=L(\mult)$. Then $\pi\mapsto \pi^t$ is the corresponding involution on $\Irr$.

\subsubsection{} For $\mult\in\OO$ let $\supp(\mult)=\{\rho\in \Cusp: \rho\in\Delta\text{ for some }\Delta\in \mult\}$ be the support of $\mult$.
(Note that  $\supp(\mult)=\supp(Z(\mult))=\supp(L(\mult))$.)

A multi-set $\mult\in \OO$ is called rigid if $\supp(\mult)\subseteq\rho^\Z$ for some $\rho\in \Cusp$.  
Let 
\[
\OO_\rho=\{\mult\in\OO:\supp(\mult)\subseteq \rho^\Z\}
\]
be the set of rigid multi-sets supported on $\rho^\Z$.

\subsubsection{}
Lemma \ref{lem: useful cusp} reduces the study of $\Sp$-distinguished representations in $\Irr$ to those supported on a cuspidal $\Z$-line.

From now on fix $\rho\in \Cusp$ once and for all. We will study $\Sp$-distinction of certain rigid representations supported on $\rho^\Z$.

\section{A necessary condition for $\Sp$-distinction of $Z(\mult)$}\label{s: nec Z}

In this section we show that if $\mult\in \OO_\rho$ is such that $Z(\mult)$ is $\Sp$-distinguished then all segments in $\mult$ are of even length.

The main tool is the geometric lemma of \S \ref{sec: gl}. We also apply a result of Heumos and Rallis that we first recall.

\subsection{Results of Heumos and Rallis} 

\subsubsection{}The following disjointness of models is \cite[Theorem 3.2.2]{MR1078382}. 
\begin{lemma}\label{gen not sp}
If $\pi\in \Irr$ is generic then it is not $\Sp$-distinguished. 
\qed
\end{lemma}

\subsubsection{} We also recall \cite[Theorem 11.1]{MR1078382} that provides first examples of irreducible, $\Sp$-distinguished representations that are not necessarily one dimensional.
\begin{lemma}\label{2speh sp} 
Let $\Delta$ be a segment in $\Cusp$. Then $L(\{\Delta,\nu\Delta\})$ is $\Sp$-distinguished. 
\end{lemma}
For the convenience of the reader, and since the proof bellow will be generalized in the sequel, we recall here the argument given by Heumos and Rallis.

Recall that $L(\{\Delta,\nu\Delta\})$ lies in an exact sequence
\[
0\rightarrow\pi \rightarrow \lambda(\{\Delta,\nu\Delta\})=\nu L(\Delta)\times L(\Delta)\rightarrow  L(\{\Delta,\nu\Delta\})\rightarrow 0
\]
where $\pi\in \Irr$ is generic. The representation $\nu L(\Delta)\times L(\Delta)$ is $\Sp$-distinguished by Lemma \ref{lem: closed}, whereas $\pi$ is not $\Sp$-distinguished by Lemma \ref{gen not sp}. Therefore, $L(\{\Delta,\nu\Delta\})$ is $\Sp$-distinguished. 

\begin{remark}\label{rmk: Speh dist}
Representations of the form $L(\{\Delta,\nu\Delta,\dots,\nu^{n-1}\Delta\})$ are often referred to as Speh representations. 
In \cite{MR2332593}, Lemma \ref{2speh sp} is generalized, showing that the Speh representation $L(\{\Delta,\nu\Delta,\dots,\nu^{n-1}\Delta\})$ is $\Sp$-distinguished if and only if $n$ is even. This characterization of $\Sp$-distinguished Speh representations was based on a global argument involving the period integrals of certain Eisenstein series. In this paper we provide a local proof of a generalization.
Since Lemmas \ref{gen not sp} and \ref{2speh sp} will be applied in the sequel, we  emphasize that their proofs in \cite{MR1078382} are purely local.
\end{remark}

\subsection{On $\Sp$-distinction of $Z(\mult)$} 
\subsubsection{} Let $\Delta=[a,b]_{(\rho)}$ be a segment in $\rho^\Z$. A special case of Remark \ref{rmk: Speh dist} says that $Z(\Delta)=L(\{\nu^a\rho\},\{\nu^{a+1}\rho\},\dots,\{\nu^b\rho\})$ is $\Sp$-distinguished if and only if $\ell(\Delta)$ is even. For the sake of completeness of our local proof, we now provide a proof of the easy part of this equivalence. A local proof of the other implication will be a part of Corollary \ref{cor: dist lad Z}.
\begin{lemma}\label{lem: Z dist}
If the representation $Z(\Delta)$ is $\Sp$-distinguished then $\ell(\Delta)$ is even. 
\end{lemma}
\begin{proof}
Assume that $Z(\Delta)$ is $\Sp$-distinguished. By Lemma \ref{drmk: ist quot}, $\nu^b\rho\times\cdots\times \nu^a\rho$ is also $\Sp$-distinguished.

Let $d$ be such that $\rho\in\Alg(G_d)$, $G=G_{d(b-a+1)}$ and $M=M_{(d,\dots,d)}$ the standard Levi subgroup of $G$ so that $\sigma=\nu^b\rho\otimes\cdots\otimes\nu^a\rho\in \Alg(M)$. By Lemma \ref{lem: ex orb}, there exists an orbit, relevant to $\sigma$.

Since $\rho$ is cuspidal, it follows from \eqref{eq: trans jm} that $\jm_{L,M}(\sigma)=0$ for every proper Levi subgroup $L$ of $M$. It therefore follows from Corollary \ref{cor: rel} that an orbit relevant to $\sigma$ is $M$-admissible. 

Since all elements of $\Cusp$ are generic, it now follows from Lemma \ref{gen not sp} and \eqref{dist cond} (for $L=M$) that there exists, in particular, an involution on $\Delta$ without fixed points. Therefore $\ell(\Delta)$ is even.
\end{proof}

\subsubsection{}\label{jm Z} This can be generalized to products of $Z(\Delta)$'s, but we first need to explicate their Jacquet modules. We start with the Jacquet module of $Z(\Delta)$ itself. 

Let $\Delta=[a,b]_{(\rho)}$ be a segment in $\Cusp$. 
Recall the description of the Jacquet module of $Z(\Delta)$ following \cite[\S3.4]{MR584084}.
Suppose that $\rho\in\Alg(G_d)$ and let $n=(b-a+1)d$ so that $Z(\Delta)\in \Alg(G_n)$. 
Let $M=M_\alpha$ for a decomposition $\alpha=(n_1,\dots,n_k)$ of $n$. Then $\jm_{M,G_n}(Z(\Delta))=0$ unless $d|n_i$, $i=1,\dots,k$ in which case
\[
\jm_{M,G_n}(Z(\Delta))=Z(\Delta_1) \otimes\cdots\otimes Z(\Delta_k)
\]
where $\Delta_i=[a_i,b_i]_{(\rho)}$, $a_1=a$, $a_{i+1}=b_i+1$, $i=1,\dots,k-1$ and $d(b_i-a_i+1)=n_i$, $i=1,\dots,k$.

\subsubsection{}\label{Z jm on M}
Suppose that $\beta$ is a refinement of a decomposition $\alpha$ of $n$ and let $M=M_\alpha$ and $L=M_\beta$. For the parts of the decompositions and the ordered index set $\indset$ we apply the notation of \S \ref{General orbits}.

Let $\Delta_1,\dots,\Delta_k$ be segments of cuspidal representations so that $\sigma=Z(\Delta_1)\otimes\cdots\otimes Z(\Delta_k)$ is an irreducible representation of $M$. It follows from \eqref{eq: trans jm} and \S \ref{jm Z} that whenever non-zero
\[
\jm_{L,M}(\sigma)=\tensor\limits_{\imath\in(\indset,\prec)} Z(\Delta_{\imath})
\]
where $\jm_{M_{\beta_i}, G_{n_i}}(Z(\Delta_i))=Z(\Delta_{i,1})\otimes\cdots\otimes Z(\Delta_{i,k_i})$ is prescribed by \S \ref{jm Z}.

\subsubsection{} 
\begin{proposition}\label{prop: Z dist}
Let 
$\mult\in \OO_\rho$. If $\tilde\zeta(\mult)$ is $\Sp$-distinguished then $Z(\Delta)$ is $\Sp$-distinguished and, in particular, $\ell(\Delta)$ is even for all $\Delta\in \mult$. 

In particular, if $Z(\mult)$ is $\Sp$-distinguished then $\ell(\Delta)$ is even for all $\Delta\in \mult$.
\end{proposition}
\begin{proof}
Fix an order $\mult=\{\Delta_1,\dots,\Delta_k\}$ so that $b(\Delta_1)\le \cdots \le b(\Delta_k)$ and note that 
$\{\Delta_k,\dots,\Delta_1\}$ is a standard order on $\mult$, i.e., that 
\[
\tilde\zeta(\mult)\simeq Z(\Delta_1)\times\cdots \times Z(\Delta_k).
\]

Let $\sigma=Z(\Delta_1)\otimes\cdots \otimes Z(\Delta_k)$, $G=G_n$ and $M=M_\alpha$ a standard Levi subgroup of $G$ such that $Z(\Delta_1)\times\cdots \times Z(\Delta_k)=\ip_{G,M}(\sigma)$. 
Assume that $\tilde\zeta(\mult)$ is $\Sp$-distinguished. By Lemma \ref{lem: ex orb} there exists $w\in {}_MW_{\inv(M)}\cap [w_{2n}]w_{2n}$ that is relevant to $\sigma$. 

We show that $Z(\Delta_i)$ is $\Sp$-distinguished for all $i$. 
Assume, by contradiction, the contrary and use the notation of \S \ref{General orbits}. 
It follows from Corollary \ref{cor: fine rel} that $\tau$ is not the trivial involution. Let $\imath\in\indset$ be minimal such that $\tau(\imath)\ne \imath$.
Then $\imath\prec \tau(\imath)$ and it follows from \eqref{eq: stronger cond} that $\imath=(i,1)$ for some $1\le i\le k$. But then the condition $\Delta_{\imath}= \nu\Delta_{\tau(\imath)}$ (from Corollary \ref{cor: fine rel}) contradicts our choice of order on $\mult$ and \S \ref{Z jm on M}. 

It therefore follows that $Z(\Delta_i)$ is $\Sp$-distinguished and from Lemma \ref{lem: Z dist} that $\ell(\Delta_i)$ is even for $i=1,\dots,k$. The last part of the proposition follows from Lemma \ref{drmk: ist quot}.
\end{proof}

Applications of the geometric lemma to the study of $\Sp$-distinction of representations of the form $\lambda(\mult)$, $\mult\in\OO$ (standard modules) 
lead to certain combinatorial problems that we are only able to solve in special cases. In the next section we formulate these problems and present the proofs for our partial results.
The section is written in such a way that it is independent of the rest of the paper and elementary. The problem we raise is accessible to every mathematician.

\section{Multi-sets of segments}\label{s: multisets}
We formulate an elementary problem on multi-sets of segments of integers that has applications to representation theory. We are only able to provide a partial solution.
\subsubsection{}\label{sss: mult}

By a multi-set $f$ of elements in a set $X$ we mean a function $f:X\rightarrow \Z_{\ge 0}$ of finite support. The support of $f$ is also referred to as the  set underlying $f$. If $f$ takes value in $\{0,1\}$ then we identify $f$ with its support and say that it is a set. 
For example, for $x\in X$ the set of one element $\{x\}$ is the characteristic function of $x$.

Denote by $\abs{f}=\sum_{x\in X} f(x)$ the size of the multi-set $f$. By abuse of notation, we sometimes write $f=\{x_1,\dots,x_t\}$ where $t=\abs{f}$ and $x\in X$ equals $x_i$ for exactly $f(x)$ indices $i$. We refer to the presentation $\{x_1,\dots,x_t\}$ as an order on $f$. 

We write $x\in f$ if $x$ is in the support of $f$.

\subsubsection{} 
By a segment of integers we mean a set $[a,b]=\{a,a+1,\dots,b\}$ where $a\le b$ are integers. By convention, the empty set is a segment. Let $\sgm$ denote the set of all segments of integers. 
Consider the operation  
\[
\nu[a,b]=[a+1,b+1],\ \ \ [a,b]\in \sgm
\]
and define the following relation on $\sgm$. For $[a,b],\,[a',b']\in \sgm$ we say that $[a,b]$ \emph{precedes} $[a',b']$ and write $[a,b]\prec[a',b']$ if $a<a'$, $b<b'$ and $b \ge a'-1$.

By a \emph{decomposition} of $[a,b]\in \sgm$ we mean a $k$-tuple of segments $([a_1,b_1],\dots,[a_k,b_k])\in \sgm^k$, $k\in \N$,  such that $b_1=b$, $a_k=a$ and $b_{i+1}=a_i-1$, $i=1,\dots,k-1$.
The decomposition is called trivial if $k=1$.

\subsubsection{}
Let $\OO_\Z$ be the set of multi-sets of segments of integers. 
We say that a multi-set $\mult=\{\Delta_1,\dots,\Delta_k\}$ is ordered in standard form if $\Delta_i\not\prec\Delta_j$ for all $i<j$.

Given an ordered multi-set $\mult=\{\Delta_1\dots,\Delta_k\}\in\OO_\Z$, by a decomposition of $\mult$ we mean a decomposition of $\Delta_i$ for all $i=1,\dots,k$. 
The decomposition is called trivial if the decomposition of each $\Delta_i$ is trivial.

It will be convenient to index a decomposition of an ordered multi-set as follows.
If $(\Delta_{i,1},\dots,\Delta_{i,k_i})$ is the decomposition of $\Delta_i$, $i=1,\dots,k$ let $(\indset,\prec)$ be the linearly ordered set 
\[
\indset=\{(i,j):i=1,\dots,k,\ j=1,\dots,k_i\}
\] 
with the lexicographic order $(i,j) \prec (i',j')$ if either $i<i'$ or $i=i'$ and $j<j'$. We further consider the partial order $(i,j)\ll(i',j')$ if $i<i'$.
Thus, a decomposition of an ordered multi-set $\mult=\{\Delta_1,\dots,\Delta_k\}\in \OO_\Z$ produces a new ordered multi-set $\{\Delta_\imath:\imath\in(\indset,\prec)\}$.


\subsubsection{Relevant decompositions}
\begin{definition}\label{rel dec}
Let $\mult=\{\Delta_1,\dots,\Delta_k\}\in\OO_\Z$ be an ordered multi-set. 
We say that an ordered decomposition $\{\Delta_\imath:\imath\in (\indset,\prec)\}$ (of the order $\{\Delta_1,\dots,\Delta_k\}$ of $\mult$) is \emph{relevant} to $\{\Delta_1,\dots,\Delta_k\}$ if there exists  an involution $\tau$ on $\indset$ satisfying the following properties:
\begin{enumerate}
\item\label{order flip} $\tau(i,j+1)\ll\tau(i,j),\ i=1,\dots,k,\,j=1,\dots,k_i-1$;
\item\label{no fix} $\tau(\imath)\ne \imath, \ \imath\in\indset$;
\item\label{segment match} $\Delta_\imath=\nu\Delta_{\tau(\imath)}$ whenever $\imath\prec\tau(\imath)$.
\end{enumerate}
\end{definition}

\subsubsection{} The involutions $\tau$ in the above definition must satisfy the following property. 
\begin{lemma}\label{lem: 1 to end}
Let $\tau$ be an involution on $\indset$ satisfying conditions \eqref{order flip} and \eqref{no fix} of Definition \ref{rel dec}. Then, there exist $i_1>\cdots>i_{k_1}>1$ such that $\tau(1,j)=(i_j,k_{i_j})$, $j=1,\dots,k_1$.
\end{lemma}
\begin{proof}
Let $\tau(1,j)=(i_j,r_j)$. The inequalities $i_1>\cdots>i_{k_1}>1$ are immediate from conditions \eqref{order flip} and \eqref{no fix} of Definition \ref{rel dec}.
If $r_j<k_{i_j}$ then, again by the same condition, $\tau(i_j,r_j+1)\ll (1,j)$ which is impossible. Therefore $r_j=k_{i_j}$.
\end{proof}

\subsubsection{Distinguished multi-set} 
\begin{definition}\label{def: dist mult}
A multi-set $\mult\in\OO_\Z$ is called distinguished if every standard order of $\mult$ admits a relevant decomposition. 
\end{definition}
\subsubsection{Speh type}\label{def: Sp type}
For $\mult=\{\Delta_1,\dots,\Delta_k\}\in\OO_\Z$ and $n\in\Z$ let $\nu^n\mult=\{\nu^n\Delta_1,\dots,\nu^n\Delta_k\}$.
It is easy to see that the following conditions are equivalent for $\mult\in\OO_\Z$: 
\begin{itemize}
\item $\mult$ is of the form $\multn+\nu\multn$ for some $\multn\in\OO_\Z$; 
\item the trivial decomposition of $\mult$ is relevant to some standard order of $\mult$; 
\item the trivial decomposition of $\mult$ is relevant to any standard order of $\mult$. 
\end{itemize}
\begin{definition}
We say that $\mult\in\OO_\Z$ is of Speh type if it satisfies the above equivalent conditions.
\end{definition}

\subsubsection{The main Hypothesis}
Consider the following property of a multi-set $\mult\in\OO_\Z$.
\begin{hypothesis}\label{hyp*}
If $\mult$ is distinguished then $\mult$ is of Speh type. 
\end{hypothesis}

Unfortunately, we do not have enough information to determine whether Hypothesis \ref{hyp*} is satisfied for all $\mult\in\OO_\Z$. We will prove it in some special cases and in particular when $\mult$ is a set.

\subsubsection{} In fact, for some of the representation theoretic applications we have in mind it suffices to consider a slightly weaker property.
For $[a,b]\in\sgm$ let $[a,b]^\vee=[-b,-a]$ and for $\mult\in\OO_\Z$ let $\mult^\vee\in\OO_\Z$ be defined by $\mult^\vee(\Delta)=\mult(\Delta^\vee)$, $\Delta\in\sgm$.

Consider the following property of a multi-set $\mult\in\OO_\Z$.
\begin{hypothesis}\label{hyp**}
If $\mult$ and $\mult^\vee$ are both distinguished then $\mult$ is of Speh type. 
\end{hypothesis}

\subsubsection{} In order to prove that a given multi-set $\mult$ satisfies the Hypothesis \ref{hyp*} it is enough to show that if $\mult$ is distinguished then the trivial decomposition is relevant to some standard order. In particular, it is enough to show that for some standard order, no non-trivial decomposition is relevant. We can only prove Hypothesis \ref{hyp*} for some special cases by proving this stronger version. For this purpose we define a certain standard order on multi-sets.

\subsubsection{}\label{orders}
For $\Delta=[a,b]$ let $b(\Delta)=a$ be the beginning and $e(\Delta)=b$ the end of $\Delta$. 

For $\mult\in \OO_\Z$ let $c_1>\cdots>c_s$ be such that we have the identity of sets $\{c_1,\dots,c_s\}=\{e(\Delta):\Delta\in \mult\}$. Let $\mult[i]\in\OO_\Z$ be defined by
\[
\mult[i](\Delta)=\begin{cases} \mult(\Delta) & e(\Delta)=c_i \\ 0 & \text{otherwise} \end{cases}
\]
for $i=1,\dots,s$ so that 
\[
\mult=\mult[1]+\cdots+\mult[s]
\]
and all segments in the support of $\mult[i]$ end at $c_i$.

Note that any linear order on $\mult$ that extends the relation, $\Delta<\Delta'$ whenever $\Delta\in \mult[i]$ and $\Delta'\in \mult[j]$ for all $1\le i<j\le s$, is in standard form. 

\subsubsection{}\label{oord}
In order to fix such an order, we need to linearly order each $\mult[i]$. We make a particular such choice recursively. 

Let $\mult[1]$ be ordered by $\mult[1]=\{\Delta_1,\dots,\Delta_k\}$  where $b(\Delta_1)\le \cdots\le b(\Delta_k)$ and set $\mult[1]'=\mult[1]$. Suppose that $\mult[1],\dots,\mult[i]$ are ordered and that $\mult[1]',\dots,\mult[i]'$ are defined for some $i<s$. We order $\mult[i+1]$ and define $\mult[i+1]'$ as follows. Set $\mult[i+1]=\{\Delta_1,\dots,\Delta_k,\Delta_1',\dots,\Delta_m'\}$ where $b(\Delta_1)\le \cdots\le b(\Delta_k)$, $b(\Delta_1')\ge \cdots \ge b(\Delta_m')$ and $\min(\mult[i+1],\nu^{-1}\mult[i]')=\{\Delta_1',\dots,\Delta_m'\}$ and let $\mult[i+1]'=\{\Delta_1,\dots,\Delta_k\}$.

\subsubsection{} We now prove that Hypothesis \ref{hyp*} holds for sets.
\begin{proposition}\label{prop: hyp set}
Let $\mult\in \OO_\Z$ be a set. Then, no non-trivial decomposition is relevant to the order on $\mult$ defined in \S \ref{oord}. In particular, Hypothesis \ref{hyp*} holds for $\mult$.
\end{proposition}
\begin{proof}
We prove the statement by induction on $\abs{\mult}$. 
If $\abs{\mult}\le 2$ then it follows from Lemma \ref{lem: 1 to end} that conditions \eqref{order flip} and \eqref{no fix} of Definition \ref{rel dec} cannot be satisfied by any non-trivial decomposition of $\mult$ and the proposition is therefore true. 

Let $\mult=\{\Delta_1,\dots,\Delta_k\}$ and assume by contradiction that there exists a non-trivial decomposition $\{\Delta_\imath:\imath\in \indset\}$ that is relevant to $\{\Delta_1,\dots,\Delta_k\}$. Let $\tau$ be the involution on $\indset$ satisfying  
properties \eqref{order flip}, \eqref{no fix}, \eqref{segment match} of Definition \ref{rel dec} and let $i_1>\cdots>i_{k_1}>1$ be given by Lemma \ref{lem: 1 to end} so that $\tau(1,j)=(i_j,k_{i_j})$, $j=1,\dots,k_1$. 

By the property of the order chosen $e(\Delta_1)\ge e(\Delta_{i_{k_1}})\ge \cdots \ge e(\Delta_{i_1})$.
The condition  \eqref{segment match} of Definition \ref{rel dec} implies that $e(\Delta_{i_1})\ge e(\Delta_1)-1$ and $b(\Delta_{i_1})>\cdots>b(\Delta_{i_{k_1}})=b(\Delta_1)-1$.  
Since the equality $e(\Delta_{i_{k_1}})=e(\Delta_1)$ contradicts the order chosen on $\mult[1]$ we must have $e(\Delta_{i_{k_1}})=e(\Delta_1)-1$, i.e.,  $\Delta_{i_{k_1}}=\nu^{-1}\Delta_1$ and in the notation of \S \ref{orders}, $c_2=c_1-1$ and $\Delta_{i_j}\in \mult[2]$ for all $j=1,\dots,k_1$. 

Since $\mult$ is a set, we get that $\Delta_{i_{k_1}}\not\in \mult[2]'$.
Taking the order chosen on $\mult[2]$ into consideration, now implies that $k_1=1=k_{i_1}$.  

Again since $\mult$ is a set, it is easy to see that the order defined on $\mult-\{\Delta_1\}-\{\Delta_{i_1}\}$ by \S \ref{oord} is that inherited from $\mult$, i.e., the order $\{\Delta_2,\dots,\Delta_{i_1-1},\Delta_{i_1+1},\dots,\Delta_k\}$. Furthermore, $\tau|_{\indset\setminus \{(1,1),(i_1,1)\}}$ shows that $\{\Delta_\imath:\imath\in \indset\setminus \{(1,1),(i_1,1)\}\}$ is a non-trivial decomposition relevant to $\{\Delta_2,\dots,\Delta_{i_1-1},\Delta_{i_1+1},\dots,\Delta_k\}$. This contradicts the induction hypothesis.
\end{proof}

\subsubsection{} The same proof gives another family of multi-sets satisfying Hypothesis \ref{hyp*}.

\begin{proposition} \label{prop: hyp 2}
Let $\mult\in \OO_\Z$ and set $\mult=\mult[1]+\cdots+\mult[s]$ as in \S \ref{orders}. Assume that $\abs{\mult[i]}\le 2$ for all $i=1,\dots,s$. Then, no non-trivial decomposition is relevant to the order on $\mult$ defined in \S \ref{oord}. In particular, Hypothesis \ref{hyp*} holds for $\mult$.
\end{proposition}
\begin{proof}
The first three paragraphs of the proof of Proposition \ref{prop: hyp set} hold for any multi-set and we apply the conclusions and the notation in this case.
Since $\abs{\mult[2]}\le 2$ we conclude that $k_1\le 2$. Furthermore, if $k_1=2$ then $\Delta_{i_2}=\nu^{-1}\Delta_1$, $\Delta_{i_1}=\nu^{-1}\Delta_{1,1}$ and $\mult[2]=\{\Delta_{i_2},\Delta_{i_1}\}$. But this shows that $\Delta_{i_2}\not\in\mult[2]'$ and contradicts the order on $\mult[2]$.

We therefore have $k_1=1=k_{i_1}$.  
Let $\multn=\mult-\{\Delta_1\}-\{\Delta_{i_1}\}$.  Again, it is easily observed that the order on $\multn$ defined in \S \ref{oord} is the one inherited by $\mult$ and the proposition follows by induction as in the last paragraphs of the proof of Proposition \ref{prop: hyp set}.
\end{proof}

\section{On $\Sp$-distinction of standard modules}\label{s: std mod sp}

The bijection $[a,b]\mapsto [a,b]_{(\rho)}$ from segments of integers to segments in $\rho^\Z$ induces a bijection on multi-sets from $\OO_\Z$ to $\OO_\rho$. We refer to this bijection as $\rho$-labeling and to its inverse as unlabeling. 

In this section we show that if $\mult\in \OO_\rho$ is such that $\lambda(\mult)$ is $\Sp$-distinguished then the unlabeling of $\mult$ is distinguished in the sense of Definition \ref{def: dist mult}. The results and Hypothesis of \S \ref{s: multisets} therefore become relevant to the study of $\Sp$-distinction of standard modules.

From now on we adopt the following convention. We say that a multi-set $\mult\in \OO_\rho$ satisfies a property defined on multi-sets in $\OO_\Z$ if its unlabeling satisfies this property.

\subsubsection{}\label{jm sqr}

Let $\Delta=[a,b]_{(\rho)}$ be a segment in $\rho^\Z$ and $\delta=L(\Delta)$. 
We first recall the description of the Jacquet module of $\delta$ following \cite[\S9.5]{MR584084}.

Suppose that $\rho$ is a representation of $G_d$ and let $n=(b-a+1)d$ so that $\delta\in \Alg(G_n)$. 
Let $M=M_\alpha$ for a decomposition $\alpha=(n_1,\dots,n_k)$ of $n$. Then $\jm_{M,G_n}(\delta)=0$ unless $d|n_i$, $i=1,\dots,k$ in which case
\[
\jm_{M,G_n}(\delta)=\delta_1 \otimes\cdots\otimes \delta_k
\]
where $\delta_i=L(\Delta_i)$, $\Delta_i=[a_i,b_i]_{(\rho)}$, $b_1=b$, $b_{i+1}=a_i-1$, $i=1,\dots,k-1$ and $d(b_i-a_i+1)=n_i$, $i=1,\dots,k$.

\subsubsection{}\label{jm on M} Suppose that $\beta$ is a refinement of a decomposition $\alpha$ of $n$ and let $M=M_\alpha$ and $L=M_\beta$. For the parts of the decompositions and the ordered index set $\indset$ we apply the notation of \S \ref{General orbits}.

Let $\delta=\delta_1\otimes\cdots\otimes\delta_k$ be an irreducible, essentially square-integrable representation of $M$ (i.e., $\delta_i=L(\Delta_i)$ for some segment of cuspidal representations, $i=1,\dots,k$). It follows from \eqref{eq: trans jm} that whenever non-zero
\[
\jm_{L,M}(\delta)=\tensor\limits_{\imath\in(\indset,\prec)} \delta_{\imath}
\]
where $\jm_{M_{\beta_i}, G_{n_i}}(\delta_i)=\delta_{i,1}\otimes\cdots\otimes\delta_{i,k_i}$ is prescribed by \S \ref{jm sqr}.

\subsubsection{}\label{sss: for dist} Let $M$ and $\delta$ be as above and $w\in {}_MW_{\inv(M)}\cap [w_{2n}]w_{2n}$. Recall that $\delta_i$ is generic $i=1,\dots,k$. In the notation of \S \ref{General orbits} it follows from \eqref{dist cond}, Lemma \ref{gen not sp} and \S \ref{jm on M} that 
\begin{equation}\label{eq: rel sqe}
w \ \text{is relevant for}\ \delta \ \text{if and only if}\ \tau(\imath)\ne\imath\ \text{for all}\ \imath\in\indset \ \text{and}\ \delta_\imath\simeq\nu\delta_{\tau(\imath)}\ \text{whenever}\ \imath\prec\tau(\imath).
\end{equation}

\subsubsection{} The following is an immediate consequence.
\begin{proposition}\label{prop: dist hyp} 
Let $\mult\in\OO_\rho$ satisfy Hypothesis \ref{hyp*}. If $\lambda(\mult)$ is $\Sp$-distinguished then $\mult$ is of Speh type. 
In particular, if $L(\mult)$ is $\Sp$-distinguished then $\mult$ is of Speh type.
\end{proposition}
\begin{proof}
If $\lambda(\mult)$ is $\Sp$-distinguished then for any standard order on $\mult=\{\Delta_1,\dots,\Delta_k\}$, the induced representation $L(\Delta_1)\times\cdots\times L(\Delta_k)$ is $\Sp$-distinguished. Therefore, by Lemma \ref{lem: ex orb}, $L(\Delta_1)\otimes\cdots\otimes L(\Delta_k)$ admits a relevant orbit. Combining the condition \ref{eq: stronger cond} with \eqref{eq: rel sqe}, \S \ref{sss: for dist} says that $\mult$ is distinguished in the sense of Definition \ref{def: dist mult}. Hypothesis \ref{hyp*} therefore implies that $\mult$ is of Speh type. 

The last part of the proposition follows from Lemma \ref{drmk: ist quot}.
\end{proof}
\subsubsection{} Combining Proposition \ref{prop: dist hyp} with Propositions \ref{prop: hyp set} and \ref{prop: hyp 2} we obtain
\begin{corollary}\label{cor: set case}
Let $\mult\in\OO_\rho$ either be a set or, in the notation of \S \ref{orders}, satisfy $\abs{\mult[i]}\le 2$ for all $i$. If $\lambda(\mult)$ is $\Sp$-distinguished then $\mult$ is of Speh type. In particular, if $L(\mult)$ is $\Sp$-distinguished then $\mult$ is of Speh type. \qed
\end{corollary}

\subsubsection{} We also point out the weaker consequence of Hypothesis \ref{hyp**}.
\begin{proposition}\label{prop: dist hyp**} 
Let $\mult\in\OO_\rho$ satisfy Hypothesis \ref{hyp**}. If $L(\mult)$ is $\Sp$-distinguished then $\mult$ is of Speh type. 
\end{proposition}
\begin{proof}
If $L(\mult)$ is $\Sp$-distinguished then $L(\mult^\vee)$ is $\Sp$-distinguished by Lemma \ref{lem: cont dist}. Therefore, by Lemma \ref{drmk: ist quot}, both $\lambda(\mult)$ and $\lambda(\mult^\vee)$ are $\Sp$-distinguished. It now follows, as in the proof of Proposition \ref{prop: dist hyp}, that both $\mult$ and $\mult^\vee$ are distinguished in the sense of Definition \ref{def: dist mult}. Hypothesis \ref{hyp**} therefore implies that $\mult$ is of Speh type. 
\end{proof}

\subsubsection{}
We end this section by providing an example which demonstrates that the necessary condition for distinction obtained by combining Propositions \ref{prop: Z dist} and \ref{prop: dist hyp**} is not sufficient. 
\begin{example}
Let $\pi=L(\mult)$ where
\[
\mult=\{[\nu^{4},\nu^{4}],[\nu^{3},\nu^{3}],[\nu^{3},\nu^{3}],[\nu^{2},\nu^{2}],[\nu,\nu^{2}],[1,\nu]\}.
\] 
By Proposition \ref{prop: hyp 2}, the multi-set $\mult$ satisfies Hypothesis \ref{hyp*}. Using the combinatorial algorithm of M\oe{}glin and Waldspurger (\cite{MR863522}), it is easy to see that 
\[
\mult^{t}=\{[\nu^{2},\nu^{3}],[\nu,\nu^{4}],[1,\nu]\}.
\]
We will now show that $\pi$ is not $\Sp$-distinguished. Assume the contrary, if possible. Let $\pi_{1}=Z([\nu^{2},\nu^{3}],[1,\nu])$ and $\pi_{2}=Z([\nu,\nu^{4}])$. The representation $\pi=Z(\mult^t)$ is the unique irreducible quotient of $\tilde\zeta(\mult^{t})$ and so it is also the unique irreducible quotient of $\pi_{1} \times \pi_{2}$. Thus, by Lemma \ref{drmk: ist quot}, $\pi_{1} \times \pi_{2}$ is $\Sp$-distinguished. Apply the notation of \S \ref{General orbits} with $k=2$ for an orbit that is relevant to $\pi_{1}\otimes \pi_{2}$ (by Lemma \ref{lem: ex orb}). Since $k=2$, note that $k_{2}\leq 2$.

From Corollary \ref{cor: fine rel} and \eqref{eq: trans jm} it follows that there exist irreducible components $\sigma_1$ of $\jm_{M_{\beta_1},G_{n_1}}(\pi_{1})$ and $\sigma_2$ of $\jm_{M_{\beta_2},G_{n_2}}(\pi_{2})$ such that writing
\[
\sigma_i=\sigma_{i,1}\otimes\cdots\otimes\sigma_{i,k_i},\ i=1,2,\ \ \ \sigma_{i,j}\in\Irr,\ j=1,\dots,k_i
\]
we have $\sigma_\imath=\nu\sigma_{\tau(\imath)}$ whenever $\imath\prec \tau(\imath)$ and $\sigma_\imath$ is $\Sp$-distinguished if $\tau(\imath)=\imath$. Also it follows from \S \ref{jm Z} that $\nu^{4}\in \supp(\sigma_{2,k_{2}})$. Since $\nu^{5}\notin \supp(\pi_{1})$, we deduce that $\tau(2,k_{2})=(2,k_{2})$. By \eqref{eq: stronger cond} it follows that $k_{2}=k_{1}=1$, and so $\tau(1,1)=(1,1)$. In other words, $Z([\nu^{2},\nu^{3}],[1,\nu])$ is $\Sp$-distinguished. Using the algorithm of M\oe{}glin and Waldspurger again we get that
\[
Z([\nu^{2},\nu^{3}],[1,\nu])\cong L([\nu^{3},\nu^{3}],[\nu,\nu^{2}],[1,1]).
\]
By Corollary \ref{cor: set case} we obtain a contradiction.
\end{example}

\section{The $\Sp$-distinguished ladder representations}
We classify $\Sp$-distinguished representations in the class of ladder representations introduced by Lapid and M\'inguez in \cite{MR3163355}.

\subsection{Distinction of ladder representations-the L aspect}
We classify $\Sp$-distinction in the class of ladder representations defined below.
\subsubsection{Ladder representations}\label{sss: ladder}
\begin{definition} Let $\rho\in \Cusp$. The set $\{\Delta_1,\dots,\Delta_k\}\in\OO_\rho$ is called a \emph{ladder} if 
\[
b(\Delta_1)>\cdots>b(\Delta_k)\ \ \  \text{and} \ \ \ e(\Delta_1)>\cdots>e(\Delta_k). 
\]

A representation $\pi\in \Irr$ is called a ladder 
representation if $\pi=L(\mult)$ where $\mult\in \OO_\rho$ is a ladder 
.
\end{definition}
Whenever we say that $\mult=\{\Delta_1,\dots,\Delta_k\}\in\OO_\rho$ is a ladder, we implicitly assume that $\mult$ is already ordered as in the definition above.


\subsubsection{}\label{sss: lm}
The following property allows us to show that certain ladder representations are $\Sp$-distinguished.
By convention, let $L([a,a-1]_{(\rho)})$ be the trivial representation of the trivial group and let $L([a,b]_{(\rho)})=0$ if $b<a-1$. Let $\mult=\{\Delta_1,\dots,\Delta_k\}\in \OO_\rho$ be a ladder, with $\Delta_i=[a_i,b_i]_{(\rho)}$ and for every $i=1,\dots,k-1$ let 
\[
\K_i=L(\Delta_1)\times \cdots\times L(\Delta_{i-1})\times L([a_{i+1},b_i]_{(\rho)})\times L([a_i,b_{i+1}]_{(\rho)})\times L(\Delta_{i+2})\times\cdots\times L(\Delta_k).
\]
(Thus, $\K_i=0$ if $a_i>b_{i+1}+1$.) By \cite[Theorem 1]{MR3163355} we have
\begin{theorem}\label{thm: main lad}
With the above notation let $\K$ be the kernel of the projection $\lambda(\mult)\rightarrow L(\mult)$. Then $\K=\sum_{i=1}^{k-1} \K_i$. \qed
\end{theorem}

\subsubsection{} The following is the characterization of $\Sp$-distinguished ladder representations.
\begin{theorem}\label{thm: dist lad}
Let $\mult=\{\Delta_1,\dots,\Delta_k\}\in \OO_\rho$ be a ladder. Then the following conditions are equivalent
\begin{enumerate}
\item $L(\mult)$ is $\Sp$-distinguished;
\item $k$ is even and $\Delta_{2i-1}=\nu\Delta_{2i}$ for all $i=1,\dots,k/2$;
\item $\mult$ is of Speh type.
\end{enumerate}
\end{theorem}
\begin{proof}
The equivalence of the last two conditions is obvious. If $L(\mult)$ is $\Sp$-distinguished then $\mult$ is of Speh type by Corollary \ref{cor: set case}. 

Assume now that $k=2m$ is even and $\Delta_{2i-1}=\nu\Delta_{2i}$ for all $i=1,\dots,m$. Let $\pi_i=L(\Delta_{2i-1},\Delta_{2i})$, $i=1,\dots,m$ and $\pi=\pi_1\times\cdots\times \pi_m$. Note that $\pi$ is a quotient of $\lambda(\mult)$. It follows from Lemma \ref{2speh sp} that $\pi_i$ is $\Sp$-distinguished for all $i=1,\dots,m$. Therefore, by Corollary \ref{open dist}, $\pi$ is $\Sp$-distinguished and by Lemma \ref{drmk: ist quot}, $\lambda(\mult)$ is $\Sp$-distinguished. 

Apply the notation of \S \ref{sss: lm}. In order to show that $L(\mult)$ is $\Sp$-distinguished it is enough to show that $\K$ is not $\Sp$-distinguished (by Lemma \ref{lem: dist comp}). By Theorem \ref{thm: main lad} it is enough to show that $\K_i$ is not $\Sp$-distinguished for all $i=1,\dots,2m-1$.

Note that 
\[
\mult_i=\{\Delta_1,\dots,\Delta_{i-1},[a_{i+1},b_i]_{(\rho)},[a_i,b_{i+1}]_{(\rho)},\Delta_{i+2},\dots,\Delta_k\}\in \OO_\rho
\]
is ordered by strictly decreasing end points and is therefore, in particular, a set and in standard form. Thus $\K_i\simeq\lambda(\mult_i)$.
By Corollary \ref{cor: set case}, it is enough to show that $\mult_i$ is not of Speh type for all $i=1,\dots,2m-1$.

Assume by contradiction that $\mult_i$ is of Speh type.  Let $\Delta_j'=\Delta_j$ for $j\ne i,\,i+1$, $\Delta_i'=[a_{i+1},b_i]_{(\rho)}$ and $\Delta_{i+1}'=[a_i,b_{i+1}]_{(\rho)}$. Since $\mult_i=\{\Delta_1',\dots,\Delta_{2m}'\}$
is ordered by strictly decreasing end points we clearly must have $\Delta_{2j-1}'=\nu\Delta_{2j}'$ for all $j=1,\dots,m$. If $i$ is odd this implies that $a_{i+1}=a_i+1$ contradicting the inequality $a_i>a_{i+1}$. If $i$ is even this implies that $a_{i-1}=a_{i+1}+1$ contradicting the fact that $a_{i-1}>a_i>a_{i+1}$ are integers. The theorem follows.
\end{proof}
\subsection{Distinction of ladder representations-the Z aspect}
The class of ladder representations is closed under Zelevinsky involution. We now reinterpret the classification above in order to characterize the ladders $\mult\in\OO_\rho$ so that $Z(\mult)$ is $\Sp$-distinguished.

\subsubsection{} Recall that in \cite{MR863522}, M\oe{}glin and Waldspurger describe a combinatorial algorithm to compute $\mult^t$ for $\mult\in \OO_\rho$. This algorithm takes a particularly simple form if $\mult$ is a ladder, as described in \cite[\S 3.2]{MR3163355}.
In particular, Lapid and M\'inguez observe that $\mult\in \OO_\rho$ is a ladder if and only if $\mult^t$ is a ladder.
Thus, $Z(\mult)$ is a ladder representation for a ladder $\mult\in\OO_\rho$ and every ladder representation is of this form for some ladder $\mult$.

In \S \ref{sss: MW alg} we give a recursive characterization of the M\oe{}glin and Waldspurger algorithm for ladders based on \cite[\S 3.2]{MR3163355}.

\subsubsection{}Note that for a ladder $\{\Delta_1,\dots,\Delta_k\}$, $\Delta_i\cap \Delta_{i+1}$ is either empty or a segment and therefore the length $\ell(\Delta_i\cap \Delta_{i+1})$ makes sense for all $i=1,\dots,k-1$. The following is another elementary observation that follows from \S \ref{sss: MW alg}. We omit the simple proof.

\begin{lemma}\label{lem: lad Z}
Let $\mult\in\OO_\rho$ be a ladder and let $\mult^t=\{\Delta_1,\dots,\Delta_k\}$ be ordered as a ladder. Then $\mult$ is of Speh type if and only if $\ell(\Delta_i)$ is even for all $i=1,\dots,k$ and $\ell(\Delta_i\cap \Delta_{i+1})$ is odd for all $1\le i\le k-1$ such that $\Delta_i\cup\Delta_{i+1}$ is a segment. \qed
\end{lemma}

\subsubsection{} Combining Lemma \ref{lem: lad Z} and Theorem \ref{thm: dist lad} we obtain another classification of $\Sp$-distinguished ladder representations.
\begin{corollary}\label{cor: dist lad Z}
Let $\rho\in \Cusp$ and $\mult=\{\Delta_1,\dots,\Delta_k\}\in\OO_\rho$ be a ladder.
Then $Z(\mult)$ is $\Sp$-distinguished if and only if we have
\begin{enumerate}
\item $\ell(\Delta_i)$ is even for all $i=1,\dots,k$ and
\item $\ell(\Delta_i\cap \Delta_{i+1})$ is odd for all $1\le i\le k-1$ such that $\Delta_{i}\cup\Delta_{i+1}$ is a segment.
\end{enumerate}
\qed
\end{corollary}

\begin{remark}
In \cite {MR2414223}, we obtained a classification of the $\Sp$-distinguished unitary dual in terms of Tadic's classification (\cite{MR870688}). Recall (see Remark \ref{rmk: Speh dist}) that a Speh representation $L(\{\Delta,\nu\Delta,\dots,\nu^{n-1}\Delta\})$ is $\Sp$-distinguished if and only if it is even (i.e. $n$ is even). Any unitary representation is a product of Speh representations and it was already proved in \cite{MR2332593} that a product of even Speh representations is $\Sp$-distinguished. The disjointness of Klyachko models obtained in \cite {MR2414223}, together with a prescribed model for any irreducible unitary representation \cite[Theorem 3.7]{MR2417789} imply that if an irreducible product of Speh representations is $\Sp$-distinguished then all Speh representations in the product are even. Based on our current results, we can reprove this implication without reference to Klyachko models as follows. If $\Delta\subseteq \rho^\Z$ for $\rho\in \Cusp$ and $\ell=\ell(\Delta)$ then it is well known that $L(\{\Delta,\nu\Delta,\dots,\nu^{n-1}\Delta\})=Z(\{\Delta',\nu\Delta,\dots,\nu^{\ell-1}\Delta'\})$ for some segment $\Delta' \subseteq \rho^\Z$ with $\ell(\Delta')=n$. If $Z(\mult)=\pi_1\times\cdots\times \pi_k\in \Irr$ is $\Sp$-distinguished and $\pi_i=Z(\mult_i)$ is a Speh representation for all $i=1,\dots,k$, then $\mult=\mult_1+\cdots+\mult_k$ and it follows from Proposition \ref{prop: Z dist} and the results of \S \ref{ss: cusp lines} that all segments in $\mult$ are of even length, i.e., that all $\pi_i$'s are even Speh representations.
\end{remark}
ֿ

\subsubsection{} Our next goal is to study the $\Sp$-distinguished representations in the class of representations in $\Irr$ that are induced from ladder representations. We only obtain a classification of this class conditional to Hypothesis \ref{hyp*}. Our proof is based on certain combinatorial statements concerning the multi-sets in $\OO_\rho$ that are obtained as sums of ladders in the above context. It is more convenient, to formulate these technical results by unlabeling. We therefore, now use the convention that $\mult\in \OO_\Z$ satisfies a property on $\OO_\rho$ if its $\rho$-labeling satisfies this property.
The next section collects the required results on multi-sets in $\OO_\Z$.

\section{On sums of ladders of Speh type}\label{s: ladder sums}

\subsubsection{}

For segments $\Delta,\,\Delta'\in \sgm$, write $\Delta\le_b \Delta'$ if either $b(\Delta)<b(\Delta')$ or $b(\Delta)=b(\Delta')$ and $e(\Delta)\le e(\Delta')$.
Thus, $\le_b$ is a linear order on $\sgm$.

\subsubsection{}

Let $\ell\in \N$ and $\mult\in \OO_\Z$ be such that $\ell(\Delta)=\ell$ for all $\Delta\in \mult$. If $\Delta\in \mult$ is minimal with respect to $\le_b$ then we can express $\mult$ as a linear combination
\[
\mult=\sum_{n=0}^N a_n\{\nu^n\Delta\}
\]
with $a_0\in \N$ and $a_n\in\Z_{\ge 0}$ for all $n=1,\dots,N$ for some large enough $N\in \N$. 
\begin{lemma}\label{lem: linear cond}
With the above assumptions and notation, if $\mult$ is of Speh type then 
\begin{equation}\label{eq: lin cond}
\sum_{n=0}^N (-1)^{N-n}a_n=0.
\end{equation}
\end{lemma}
\begin{proof}
Note that
\[
\mult=\sum_{n=0}^{N-1} b_n(\{\nu^n\Delta\}+\{\nu^{n+1}\Delta\})+b_N \{\nu^N\Delta\}
\]
where
$b_0=a_0$ and $b_n=\sum_{i=0}^n(-1)^{n-i}a_i$. If $\mult$ is of Speh type, then it follows that $b_n\ge 0$ for all $n$ and that, as required, $b_N=0$.
\end{proof}

\subsubsection{}\label{sss: MW alg} As observed by M{\oe}glin and Waldspurger the Zelevinsky involution $\pi^t=L(\mult^t)$ of a rigid representation $\pi=L(\mult)\in \Irr$ where $\mult\in \OO_\rho$ is `blind' to the cuspidal line $\rho^\Z$ where it is supported. Indeed, the Moeglin-Waldspurger algorithm defines an involution on $\OO_\Z$ that gives $\mult\mapsto \mult^t$ via $\rho$-labeling. We now explicate a recursive characterization of the Moeglin-Waldspurger algorithm of the Zelevinsky involution for ladders based on \cite[\S 3.2]{MR3163355}. 

For $\Delta=[a,b]\in \sgm$ we have $\{\Delta\}^t=\sum_{c=a}^b \{c\}=\{\{b\},\dots,\{a\}\}$.

Let $\mult=\{\Delta_1,\dots,\Delta_k,\Delta_{k+1}\}$ be a ladder and let $\mult'=\{\Delta_1,\dots,\Delta_k\}$. Write $\Delta_i=[a_i,b_i]$. 
Let $(\mult')^t=\{\Delta_1',\dots,\Delta_s'\}$ be ordered as a ladder. 

If $b_{k+1}+1<a_k$ then $\mult^t=(\mult')^t+\{\Delta_{k+1}\}^t$ and therefore 
\[
\mult=\{\Delta_1',\dots,\Delta_s',\{b_{k+1}\},\dots,\{a_{k+1}\}\} 
\]
is ordered as a ladder.
Otherwise, let $c=s-(b_{k+1}-a_k+2)$. Then $c\ge 0$ and  
\begin{equation}\label{eq: z inv l}
\mult^t=\{\Delta_1',\dots, \Delta_c',{}^+\Delta_{c+1}',\dots,{}^+\Delta_s',\{a_k-2\},\dots,\{a_{k+1}\}\}
\end{equation}
where ${}^+[a,b]=[a-1,b]$.

In other words, in order to obtain the ladder $\mult^t$ from $(\mult')^t$ one has to perform the following steps.
If $\Delta_{k+1}\not\prec\Delta_k$ then add to $\mult'$ at the tail of the ladder, the ladder $\{\Delta_{k+1}\}^t$, i.e., the $\ell(\Delta_{k+1})$ length one segments consisting of elements of $\Delta_{k+1}$ in decreasing order. Otherwise, $b_{k+1}-a_k+2\ge 1$. Starting with $(\mult')^t$, replace $\Delta$ by ${}^+\Delta$ for each of the last $b_{k+1}-a_k+2$ segments of $(\mult')^t$ and then add at the tail of the resulting ladder, the $a_k-a_{k+1}-1$ length one segments consisting of elements of $[a_{k+1},a_k-2]$ in decreasing order. 

\subsubsection{} Let $\mult_1,\dots,\mult_k\in\OO_\Z$ be ladders and let $\Delta_0$ be minimal in $\mult=\mult_1+\cdots+\mult_k$ with respect to $\le_b$. Let $\mult_i^\dag=\mult_i$ if $\Delta_0\not\in\mult_i$ and $\mult_i^\dag=\mult_i-\{\Delta_0\}$ otherwise. 

\begin{lemma}\label{lem: t speh}
With the above notation, suppose that the ladder $(\mult_i^\dag)^t$ is of Speh type for all $i=1,\dots,k$ and that $\ell(\Delta_0)$ is even. If $\multn=\mult_1^t+\cdots+\mult_k^t$ is of Speh type then $\mult_i^t$ is of Speh type for all $i=1,\dots,k$.
\end{lemma}
\begin{proof}
If $\mult_i=\mult_i^\dag$ then $\mult_i^t$ is of Speh type. If $\mult_i\ne\mult_i^\dag$, let $\Delta_i$ be the minimal segment in $\mult_i^\dag$ with respect to $\le_b$. If $\Delta_0\not\prec\Delta_i$ then it follows from \S \ref{sss: MW alg} that $\mult_i^t=(\mult_i^\dag)^t+\{\Delta_0\}^t$. Since $\ell(\Delta_0)$ is even it follows that $\mult_i^t$ is of Speh type. If $\Delta_0 \prec \Delta_i$, then $\Delta_0\cup\Delta_i$ is a segment and by Lemma \ref{lem: lad Z} (applied once with $\mult=\mult_i^t$ and once with $\mult=(\mult_i^\dag)^t$), $\mult_i^t$ is of Speh type if and only if $\ell(\Delta_0\cap \Delta_i)$ is odd. 

Let $A=\{1\le i\le k: \mult_i^t \ \text{is not of Speh type}\}$ and
$\multn=\multn_1+\multn_2$ where $\multn_1=\sum_{i\in A}\mult_i^t$.
Assume by contradiction, that $\multn$ is of Speh type and $\multn_1\ne 0$ (i.e. $A\ne \emptyset$).

Let $i\in A$ and $c_i=s_i-(e(\Delta_0)-b(\Delta_i)+2)$ where $s_i=\abs{(\mult_i^\dag)^t}$. 
Note that by the above remarks, $s_i-c_i=e(\Delta_0)-b(\Delta_i)+2=\ell(\Delta_0\cap \Delta_i)+1\ge 1$ is odd. Since $(\mult_i)^t$ is of Speh type, $s_i$ is even, hence $c_i$ is also odd. Furthermore, $d_i=b(\Delta_i)-b(\Delta_0)-1$ is odd and by \eqref{eq: z inv l} $\mult_i^t$ ordered as a ladder has the form 
\begin{multline}
\mult_i^t=\{\nu\Delta_1',\Delta_1',\dots, \nu\Delta_{(c_i-1)/2}',\Delta_{(c_i-1)/2}',\nu\Delta_{(c_i+1)/2}',{}^+\Delta_{(c_i+1)/2}',\\ \nu\Delta_1'',\Delta_1'',\dots,\nu\Delta_{(s_i-c_i-1)/2}'',\Delta_{(s_i-c_i-1)/2}'',\{x_0+d_i-1\},\dots,\{x_0+1\},\{x_0\}\}
\end{multline}
where $\ell(\Delta_i'')>1$ for $i=1,\dots,(s_i-c_i-1)/2$ and $x_0=b(\Delta_0)$.
Note further that $b({}^+\Delta_{(c_i+1)/2}'))=e(\Delta_0)$ is independent of $i\in A$. Thus, we may decompose
\[
\mult_i^t=\mathfrak{a}_i+\{\nu\Gamma_i,{}^+\Gamma_i\}+\mathfrak{b}_i+\{x_0\}
\]
where $\mathfrak{a_i}$ and $ \mathfrak{b}_i$ are of Speh type, $b(\Delta)>e(\Delta_0)+2$ for all $\Delta\in \mathfrak{a}_i$, $b(\Delta)<e(\Delta_0)$ for all $\Delta\in \mathfrak{b}_i$ and $\Gamma_i\in\sgm$ is such that $b({}^+\Gamma_i)=e(\Delta_0)$ and therefore also $b(\nu\Gamma_i)=e(\Delta_0)+2$. In particular, $b(\Delta)\ne e(\Delta_0)+1$ for all $\Delta\in \mult_i^t$. 

For $\ell\in \N$ and a multi-set $\mathfrak{a}\in \OO_\Z$ let $\mathfrak{a}(\ell)=\delta(\ell) \cdot\mathfrak{a}$ where $\delta(\ell)$ is the characteristic function of all segments of length $\ell$. Clearly, $\mathfrak{a}$ is of Speh type if and only if $\mathfrak{a}(\ell)$ is of Speh type for all $\ell\in\N$.

Fix $i_0\in A$ and let $\ell=\ell({}^+\Gamma_{i_0})>1$, $B=\{i\in A:\ell=\ell({}^+\Gamma_i)\}$ and $C=\{i\in A:\ell=\ell(\nu\Gamma_i)\}$. Since $\ell({}^+\Gamma_i)=\ell(\nu\Gamma_i)+1$, $B$ and $C$ are disjoint. We have
\[
\mult_i^t(\ell)=\begin{cases} \mathfrak{a}_i(\ell)+\{{}^+\Gamma_i\}+\mathfrak{b}_i(\ell) & i\in B \\  \mathfrak{a}_i(\ell)+\{\nu\Gamma_i\}+\mathfrak{b}_i(\ell) & i\in C \\  \mathfrak{a}_i(\ell)+\mathfrak{b}_i(\ell) & i\in A\setminus (B\cup C). \end{cases}
\]
Set $\Delta={}^+\Gamma_{i_0}$ and note further that for $i\in B$ we have ${}^+\Gamma_i=\Delta$ and for $i\in C$ we have $\nu\Gamma_i=\nu^2\Delta$. It follows that 
\[
\multn(\ell)=\multn_1(\ell)+\multn_2(\ell)=(\abs{B}\{\Delta\}+\abs{C}\{\nu^2\Delta\})+\sum_{i\in A} ( \mathfrak{a}_i(\ell)+\mathfrak{b}_i(\ell) )+\multn_2(\ell).
\] 
By assumption $\multn(\ell)$ and $\sum_{i\in A} ( \mathfrak{a}_i(\ell)+\mathfrak{b}_i(\ell) )+\multn_2(\ell)$ are both of Speh type and therefore, by Lemma \ref{lem: linear cond}, each of them satisfies the linear condition \eqref{eq: lin cond}. It follows that $\abs{B}\{\Delta\}+\abs{C}\{\nu^2\Delta\}$ satisfies the same linear condition, i.e., that $\abs{B}+\abs{C}=0$. But since $i_0\in B$ this is a contradiction. 
\end{proof}

\section{On Distinction of representations induced from ladder}\label{s: dist ladder}

We now study distinction in the class of representations in $\Irr$ that are induced from ladder representations. For a product of more then two ladder representations, our results are only conditional on Hypothesis \ref{hyp**}. 


\subsubsection{} We recall \cite[Lemma 5.17]{1411.6310}. It reduces the reducibility of a product of ladder representations to induction from a maximal parabolic.

\begin{lemma}\label{lem: red max}
Let $\pi_1,\dots,\pi_k$ be ladder representations and let $\pi=\pi_1\times\cdots\times \pi_k$. Then $\pi\in \Irr$ if and only if $\pi_i\times \pi_j\in\Irr$ for all $i< j$. \qed
\end{lemma}

\subsubsection{}\label{sss: NC} For a maximal parabolic, the following criterion is a combination of \cite[Proposition 5.21 and Lemma 5.22]{1411.6310}.

\begin{definition}\label{def: NC}
Let $\mult=\{\Delta_1,\dots,\Delta_t\}$ and $\multn=\{\Delta_1',\dots,\Delta_s'\}$ be ladders in $\OO_\rho$.
We say that the condition $NC(\mult,\multn)$ is satisfied if 
there exist
$k\ge 0$, $1 \le i \le t$ and $1 \le j \leq s$ such that $i+k\le t$, $j+k\le s$ and the following properties are satisfied:
\begin{enumerate}
\item\label{1} $\Delta_{i+l}\prec\Delta'_{j+l}$ for all $l=0,\dots,k$;
\item\label{2} $\nu^{-1}\Delta_{i-1}\not\prec \Delta'_j$ if $i>1$; 
\item\label{3} $\nu^{-1}\Delta_{i+k}\not\prec\Delta'_{j+k+1}$ if $j+k+1\le s$.
\end{enumerate}
\end{definition}
\begin{proposition}\label{prop: irr max}
In the above notation $Z(\mult)\times Z(\multn)\in\Irr$ if and only if neither $NC(\mult,\multn)$ nor $NC(\multn,\mult)$ hold.
\end{proposition}

The main result of this section requires some preparation.

\subsubsection{}

For segments $\Delta,\,\Delta'\in \OO_\rho$, write $\Delta\le_b \Delta'$ if either $b(\Delta)<b(\Delta')$ or $b(\Delta)=b(\Delta')$ and $e(\Delta)\le e(\Delta')$.
Thus, $\le_b$ is a linear order on $\OO_\rho$.

\begin{lemma}\label{lem: excision}
Let $\mult_1,\dots,\mult_k\in\OO_\rho$ be ladders such that $Z(\mult_1)\times\cdots\times Z(\mult_k)\in\Irr$. Let $\Delta_0$ be a minimal segment in $\mult_1+\cdots+\mult_k$ with respect to $\le_b$ and let
\[
\mult_i^\dag=\begin{cases} \mult_i-\{\Delta_0\} & \Delta_0\in \mult_i \\ \mult_i & \text{otherwise.} \end{cases}
\]   
Then $Z(\mult_1^\dag)\times\cdots\times Z(\mult_k^\dag)\in\Irr$.
\end{lemma}
\begin{proof}

By Lemma \ref{lem: red max}, it is enough to prove the lemma for the case $k=2$. Note that if $\Delta_{0}\notin \mult_{i},\ i=1,2$, then $\mult_{i}=\mult_{i}^\dag$ and we have nothing to prove. Thus we assume that $\Delta_{0}$ belongs to at least one of the two multi-sets. By the symmetry of the irreducibility criterion of Proposition \ref{prop: irr max}, we may assume without loss of generality that $\Delta_0\in \mult_1$. Write $\mult_1^\dag=\{\Delta_1,\dots,\Delta_t\}$ and $\mult_2^\dag=\{\Delta_1',\dots,\Delta_s'\}$ ordered as ladders.
Note that $\mult_1=\{\Delta_1,\dots,\Delta_t,\Delta_0\}$ is ordered as a ladder.

Assume by contradiction that $Z(\mult_1^\dag)\times Z(\mult_2^\dag)$ reduces. It follows from Proposition \ref{prop: irr max} that either $NC(\mult_1^\dag,\mult_2^\dag)$ or $NC(\mult_2^\dag,\mult_1^\dag)$ is satisfied. 

We separate into two cases and show that in each case this implies that either $NC(\mult_1,\mult_2)$ or $NC(\mult_2,\mult_1)$ holds.
Since this is a contradiction to Proposition \ref{prop: irr max} the lemma will follow.

Consider first the case that $\Delta_0\not\in \mult_2$ (i.e., $\mult_2=\mult_2^\dag$) and $NC(\mult_1^\dag,\mult_2^\dag)$ holds. Let $i,\,j,\,k$ be the indices satisfying \eqref{1}-\eqref{3} of Definition \ref{def: NC} for $(\mult_1^\dag,\mult_2^\dag)$. Then the same indices $i,\,j,\,k$ show that $NC(\mult_1,\mult_2)$ holds.

Next consider the case $\Delta_0\in \mult_2$ or $NC(\mult_1^\dag,\mult_2^\dag)$ doesn't hold.
If $\Delta_0\not\in \mult_2$ then by assumption $NC(\mult_2^\dag,\mult_1^\dag)$ holds. 
If $\Delta_0\in \mult_2$ then by symmetry between $\mult_1$ and $\mult_2$, without loss of generality, we may also assume that $NC(\mult_2^\dag,\mult_1^\dag)$ holds. 

If the indices $i,\,j,\,k$ satisfy \eqref{1}-\eqref{3} of Definition \ref{def: NC} for $(\mult_2^\dag,\mult_1^\dag)$ then for the same indices \eqref{1} and \eqref{2} are automatic for $(\mult_2,\mult_1)$, while \eqref{3} is automatic unless $j+k=t$ in which case \eqref{3} reads $\nu^{-1}\Delta_{i+k}' \not\prec \Delta_0$. By the minimality of $\Delta_0$ and since $\Delta_0\not\in \mult_2^\dag$ it follows that $\nu^{-1}\Delta'\not\prec \Delta_0$ for all $\Delta'\in\mult_2^\dag$ and in particular for $\Delta'=\Delta_{i+k}'$. It therefore follows that $NC(\mult_2,\mult_1)$ holds.
\end{proof}

\subsubsection{} \label{sss: li}

Define an operation $\mult\mapsto \mult'$ on $\OO_\rho$ as follows. For $\mult\in\OO_\rho$, let $\Delta_0$ be minimal in $\mult$ with respect to $\le_b$ and $\mult'=\mult-\mult(\Delta_0)\{\Delta_0\}$.

\begin{proposition}\label{prop: li}
Let $\Omega\subseteq \OO_\rho$ be a subset closed under the operation $\mult\mapsto \mult'$ and such that Hypothesis \ref{hyp**} holds for $\mult^t$ for all $\mult\in\Omega$.
Let $\pi_1,\dots,\pi_k$ be ladder representations such that $\pi=\pi_1\times\cdots\times\pi_k\in \Irr$. 
Let $\mult\in\OO_\rho$ be such that $\pi=Z(\mult)$ and assume that $\mult\in\Omega$.  
If $\pi$ is $\Sp$-distinguished then $\pi_i$ is $\Sp$-distinguished for all $i=1,\dots,k$.
\end{proposition}
\begin{proof}
Let $\mult_1,\dots,\mult_k$ be ladders such that $\pi_i=Z(\mult_i)$, $i=1,\dots,k$. Since $\pi$ is irreducible, we have $\mult=\mult_1+\cdots+\mult_k$. 

The proof is by induction on $\abs{\mult}$. For $\mult=0$ there is nothing to prove.
Let $\Delta_0$ be the minimal segment in $\mult$ with respect to $\le_b$. 
Let $\multn_0=\mult(\Delta_0)\{\Delta_0\}$ so that $\mult=\mult'+\multn_0$ and let $\mult_i^\dag=\min\{\mult_i,\mult'\}$, $i=1,\dots,k$. Note that 
\[
\mult_i^\dag=\begin{cases}\mult_i' & \Delta_0\in \mult_i \\ \mult_i & \text{otherwise}\end{cases}
\]
and $\mult'=\mult_1^\dag+\cdots+\mult_k^\dag$. 
By the definition of $\mult$, $\pi$ is the unique irreducible quotient of $\tilde\zeta(\mult)$. Since no segment in $\mult'$ precedes $\Delta_0$ we have $\tilde\zeta(\mult)=\zeta(\multn_0) \times \tilde\zeta(\mult')$ and therefore $\pi$ is also the unique irreducible quotient of $Z(\multn_0) \times Z(\mult')$.
Thus, by Lemma \ref{drmk: ist quot}, $Z(\multn_0) \times Z(\mult')$ is $\Sp$-distinguished. Apply the notation of \S \ref{General orbits} with $k=2$ for an orbit that is relevant to $Z(\multn_0) \otimes Z(\mult')$ (by Lemma \ref{lem: ex orb}). 

It follows from Corollary \ref{cor: fine rel} and \eqref{eq: trans jm} that there exist irreducible components $\sigma_1$ of $\jm_{M_{\beta_1},G_{n_1}}(Z(\multn_0))$ and $\sigma_2$ of $\jm_{M_{\beta_2},G_{n_2}}(Z(\mult'))$ such that writing
\[
\sigma_i=\sigma_{i,1}\otimes\cdots\otimes\sigma_{i,k_i},\ i=1,2,\ \ \ \sigma_{i,j}\in\Irr,\ j=1,\dots,k_i
\]
we have $\sigma_\imath=\nu\sigma_{\tau(\imath)}$ whenever $\imath\prec \tau(\imath)$ and $\sigma_\imath$ is $\Sp$-distinguished if $\tau(\imath)=\imath$.
By \cite[Theorem 4.2]{MR584084} we have
\[
Z(\multn_0)=\overbrace{\operatorname{Z(\Delta_0)\times\cdots\times Z(\Delta_0)}}\limits^{\mult(\Delta_0)-\text{times}}.
\]
In particular, it follows from the geometric lemma of Bernstein and Zelevinsky \cite[\S2.12]{MR0579172} and \S \ref{Z jm on M} that $\sigma_{1,1}=Z(\Delta)$ for some segment satisfying $b(\Delta)=b(\Delta_0)$. If $\tau(1,1)\ne (1,1)$ then, by \eqref{eq: stronger cond} and the fact that $k=2$, we must have $\tau(1,1)=(2,k_2)$ and therefore $\nu^{-1}\sigma_{1,1}=\sigma_{2,k_2}$. But since $\supp(\sigma_{(2,k_2)})\subseteq \supp(\sigma_2)\subseteq\supp(Z(\mult'))$ (see \eqref{eq: supp jm}) we have $\nu^{-1}b(\Delta_0)\in \supp(\nu^{-1}\sigma_{(1,1)})\setminus  \supp(\sigma_{(2,k_2)})$ which is a contradiction. It follows that $\tau(1,1)=(1,1)$ and since $k=2$, $\tau$ must be trivial. 

In other words, both $Z(\multn_0)$ and $Z(\mult')$ are $\Sp$-distinguished.
It follows from Proposition \ref{prop: Z dist} that $\ell(\Delta_0)$ is even. Let $\pi_i'=Z(\mult_i^\dag)$, $i=1,\dots,k$. It follows from Lemma \ref{lem: excision} that $\pi'=\pi_1'\times\cdots\times \pi_k'\in\Irr$ and therefore $\pi'=Z(\mult')$. By the assumption on $\Omega$ and the induction hypothesis, $\pi_i'=L((\mult_i^\dag)^t)$ is $\Sp$-distinguished for all $i=1,\dots,k$. 
By Theorem \ref{thm: dist lad}, $(\mult_i^\dag)^t$ is of Speh type for $i=1,\dots,k$. Since $\pi=L(\mult^t)$ is $\Sp$-distinguished and $\mult^t$ satisfies Hypothesis \ref{hyp**}, by Proposition \ref{prop: dist hyp**}, $\mult^t$ is of Speh type. It now follows from 
Lemma \ref{lem: t speh} that $\mult_i^t$ is of Speh type and therefore, again by Theorem \ref{thm: dist lad}, that $\pi_i=L(\mult_i^t)$ is $\Sp$-distinguished for all $i=1,\dots,k$.
\end{proof}

\subsubsection{}

Let $\Omega_k$ be the set of all $\mult\in\OO_\rho$ that are obtained as sums of at most $k$ ladders, i.e. $\mult=\mult_1+\cdots+\mult_k$ where $\mult_i$ is either zero or a ladder, and such that $Z(\mult_1)\times\cdots \times Z(\mult_k)\in \Irr$.

Since both ladder representations and $\Irr$ are closed under Zelevinsky involution and in the Grothendick group it is multiplicative, it follows that $\Omega_k$ is closed under Zelevinsky involution. It further follow from Lemma \ref{lem: excision} that $\Omega_k$ is closed under the operation $\mult\mapsto \mult'$ defined in \S \ref{sss: li}.

For a product of two ladder representations this gives the following unconditional result. 

\begin{corollary}\label{cor: prd two}
Let $\pi_1$ and $\pi_2$ be ladder representations such that $\pi=\pi_1\times\pi_2\in\Irr$. 
If $\pi$ is $\Sp$-distinguished then $\pi_1$ and $\pi_2$ are $\Sp$-distinguished.
\end{corollary}
\begin{proof}
Note that $\abs{\mult[i]}\le 2$ for all $i$ and all $\mult\in\Omega_2$. Since, as remarked above, $\Omega_2$ is closed under Zelevinsky involution, it follows from Proposition \ref{prop: hyp 2} that $\mult^t$ satisfies Hypothesis \ref{hyp*} for all $\mult\in\Omega_2$. Since we also observed above that $\Omega_2$ is closed under the operation $\mult\mapsto \mult'$ defined in \S \ref{sss: li}, the statement follows from Proposition \ref{prop: li}.
\end{proof}

\subsubsection{}
We conclude this section with an example of a family of imprimitive, $\Sp$-distinguished representations that are not ladders.
\begin{definition}\label{def: non lad ex}
Let $\mathcal F$ denote a set of irreducible representations of the form $Z(\mult)$ such that the multi-set $\mult=\{\Delta_{1},\Delta_{2},\Delta_{3}\}$ satisfies the following properties
\begin{enumerate}
\item $\ell(\Delta_{i})$ is even for all $i$,
\item $\Delta_{1}\subseteq \nu\Delta_{2}$ and $\Delta_{1}\subseteq \nu^{-1}\Delta_{2}$,
\item $\ell(\Delta_{3}\cap \Delta_{1})$ and $\ell(\Delta_{3}\cap \Delta_{2})$ are both odd.
\end{enumerate}
\end{definition}
Note that (2) implies that $\Delta_{1}\subseteq \Delta_{2}$ which in particular implies that none of these representations are ladders. Further note that $\mathcal F$ consists of only rigid representations and $\pi\in \mathcal F$ if and only if $\pi^{\vee}\in \mathcal F$. The conditions on the length of the segments in $\mult$ imply that the pairs $\{\Delta_{1},\Delta_{3}\}$ and $\{\Delta_{2},\Delta_{3}\}$ are linked. A simple example of a representation in $\mathcal F$ is $Z([\nu^{3},\nu^{4}],[\nu,\nu^{6}],[1,\nu^{3}])$. 

The next lemma shows that indeed any representation in $\mathcal F$ has the desired properties. Before we proceed, recall that an elementary operation on an arbitrary multi-set $\mult$ is to choose a pair of linked segments in it and replace the pair by their union and their intersection. By \cite[Theorem 7.1]{MR584084} any irreducible subquotient of $\zeta(\mult)$ is of the form $Z(\multn)$ where $\multn$ is a multi-set obtained from $\mult$ by a sequence of elementary operations on it.
\begin{lemma}
Let $\pi\in \mathcal F$. Then $\pi$ is $\Sp$-distinguished and imprimitive.
\end{lemma}
\begin{proof}
Let $\mult=\{\Delta_{1},\Delta_{2},\Delta_{3}\}$ be a multi-set satisfying the conditions (1), (2) and (3) of Definition \ref{def: non lad ex} and $\pi=Z(\mult)$. By Corollary  \ref{cor: dist lad Z}, $Z(\Delta_i)$ is $\Sp$-distinguished, $i=1,2,3$ and therefore by Corollary \ref{open dist}, 
\[
{\rm I}(\mult)=Z(\Delta_{1})\times Z(\Delta_{2})\times Z(\Delta_{3})
\]
is $\Sp$-distinguished. 

Applying \cite[Theorem 1.9]{MR584084} to $\zeta(\mult)$ we get that $\pi$ occurs as a subquotient of ${\rm I}(\mult)$. We now analyze the other possible irreducible subquotients of ${\rm I}(\mult)$ using \cite[Theorem 1.9 and Theorem 7.1]{MR584084}. Since $\Delta_{1}\subseteq \Delta_{2}$, any elementary operation on $\mult$ is performed on either $\{\Delta_{1},\Delta_{3}\}$ or on $\{\Delta_{2},\Delta_{3}\}$. The result will respectively contain either $\Delta_{1}\cap\Delta_{3}$ or $\Delta_{2}\cup\Delta_{3}$,  which are of odd length.
Since the first is contained in and the second contains all three segments any further sequence of operations will result in a multi-set containing one of them. Thus by Proposition \ref{prop: Z dist} none of these subquotients are $\Sp$-distinguished. The $\Sp$-distinction of $\pi$ now follows from Lemma \ref{lem: dist comp}.

Next we show that $\pi$ is imprimitive. Assume, if possible, that it is not so. Then there exists indices $i,j,k$ such that $\{i,j,k\}=\{1,2,3\}$ and $\pi\cong Z(\Delta_{i})\times Z(\Delta_{j},\Delta_{k})$. By considering the multi-set $\mult^{\vee}$ instead of $\mult$ if required, assume further that $\Delta_{3}\prec \Delta_{2}$ and hence $\Delta_{3}\prec \Delta_{1}$. Note that $Z(\Delta_{1},\Delta_{2})\times Z(\Delta_{3})\cong {\rm I}(\mult)$ which is reducible by \cite[Theorem 4.2]{MR584084}. Thus $\pi \cong Z(\Delta_{i},\Delta_{3})\times Z(\Delta_{j})$ where $\{i,j\}=\{1,2\}$. It follows from Proposition \ref{prop: irr max} and condition (2) of Definition \ref{def: non lad ex} that this product is reducible which is a contradiction. 
\end{proof}

\section{On distinction by Klyachko subgroups}\label{s: Kl}

We continue the study of Klyachko models for representations of $\GL_n(F)$, following \cite{MR1078382}, \cite{MR2417789} and \cite{MR2414223}.
Over finite fields Klyachko models were introduced in \cite{MR691984}. In that case, it is a disjoint family of models and their direct sum contains every irreducible representation with multiplicity one \cite{MR1129515}. 

Over a non-archimedean field, Heumos and Rallis observed that some representations do not admit a Klyachko model and classified those in the unitary dual that do in low rank cases (for $n\le 4$). The second and the third authors showed that the direct sum of all Klyachko models is multiplicity free and prescribed a model to any representation in the unitary dual. 

In this section, we reduce the study of Klyachko models on the admissible dual to rigid representations and prove that models behave well with respect to parabolic induction. 

\subsection{The Klyachko model setting}
\subsubsection{}
Let $G=G_n$. For a decomposition $n=2k+r$ let
\[
H_{2k,r}=\{\begin{pmatrix} h & X \\ 0 & u \end{pmatrix}: h\in \Sp_{2k}(F),\,X\in M_{2k\times r}(F),\,u\in N_r\}
\]
and $\psi=\psi_{2k,r}$ be defined by
\[
\psi(\begin{pmatrix} h & X \\ 0 & u \end{pmatrix})=\psi(u).
\]
(See \S \ref{sss: gen} for the definition of $N_r$ and its character $\psi$.)

For any $\pi\in\Alg$, being $(H_{2k,r},\psi)$-distinguished is independent of the choice of non-trivial character $\psi$ of $F$. Indeed, for any other character $\psi'\ne 1$ there is a diagonal matrix $a\in G$ normalizing $H_{2k,r}$ such that $\psi'_{2k,r}(h)=\psi_{2k,r}(ah a^{-1})$, $h\in H_{2k,r}$.

\subsubsection{}
Let $\tau$ be the involution on $G$ defined by $g^\tau=w^{-1}{}^tg^{-1}w$ where $w=\sm{0}{I_r}{I_{2k}}{0}$ and let 
\[
H_{r,2k}'=H_{2k,r}^\tau=\{\begin{pmatrix} u & X \\ 0 & h \end{pmatrix}: h\in \Sp_{2k}(F),\,X\in M_{r\times 2k}(F),\,u\in N_r\}.
\]
In \cite{MR1078382}, \cite{MR2417789} and \cite{MR2414223} we studied distinction by $(H_{r,2k}',\psi)$. Clearly, $\pi\in\Alg$ is $(H_{2k,r},\psi)$-distinguished if and only if $\pi^\tau$ is $(H_{r,2k}',\psi)$-distinguished. If $\pi\in \Irr$ then $\pi^\tau\simeq\pi^\vee$, by \cite{MR0404534}, and we get a natural isomorphism
\begin{equation}\label{eq: 2 models}
\Hom_{H_{2k,r}}(\pi,\psi)\simeq \Hom_{H_{r,2k}'}(\pi^\vee,\psi).
\end{equation}
In particular, $\pi$ is $(H_{2k,r},\psi)$-distinguished if and only if $\pi^\vee$ is $(H_{r,2k}',\psi)$-distinguished.
More generally, if $\pi_1,\dots,\pi_t\in\Irr$ then $(\pi_1\times\cdots\times\pi_t)^\tau\simeq \pi_t^\vee\times \cdots \times \pi_1^\vee$
and therefore 
\begin{equation}\label{eq: 2 models ind}
\Hom_{H_{2k,r}}(\pi_1\times\cdots\times\pi_t,\psi)\simeq \Hom_{H_{r,2k}'}(\pi_t^\vee\times \cdots \times \pi_1^\vee,\psi).
\end{equation}
In particular,
$\pi_1\times\cdots\times\pi_t$ is $(H_{2k,r},\psi)$-distinguished if and only if $\pi_t^\vee\times \cdots \times \pi_1^\vee$ is $(H_{r,2k}',\psi)$-distinguished.
\begin{remark}
We remark that the proof of \cite[Theorem 3.7]{MR2417789} applied \cite[Lemma 3.1]{MR2417789}, where \eqref{eq: 2 models} was mistakenly formulated for any representation $\pi$. The isomorphism \eqref{eq: 2 models ind} suffices to fill the gap. In any case, we provide in the sequel an independent generalization of \cite[Theorem 3.7]{MR2417789}.
\end{remark}

\subsubsection{}
Let $\pi\in\Irr\cap \Alg(G_n)$. If $\pi$ is $(H_{2k,r},\psi)$-distinguished for some decomposition $n=2k+r$ then by \eqref{eq: frob rec} it imbeds in $\Ind_{H_{2k,r}}^{G_n}(\psi)$ and we say that it admits a Klyachko model. 

By the uniqueness and disjointness of Klyachko models,  \cite[Theorem 1]{MR2414223}, both the imbedding (up to a constant multiple) and the decomposition $n=2k+r$ are uniquely determined by $\pi$. (Indeed, in the main result on distributions \cite[Proposition 1]{MR2414223} implying \cite[Theorem 1]{MR2414223} $H_{2k,r}$ and $H'_{r,2k}$ are in symmetry.)  In that case we denote by 
\[
r(\pi)=r
\] 
the Klyachko type of $\pi$.

\subsubsection{}
The main tool in our study of Klyachko models is the theory of derivatives of representations of $G_n$ developed in \cite{MR0404534}, \cite{MR0425030}, \cite{MR0579172} and \cite{MR584084}. It allows a reduction of many of the problems concerned with Klyachko models to the study of $\Sp$-distinction and generic representations.

For $\pi\in \Alg(G_n)$ and any $r=0,1,\dots,n$ we denote by $\pi^{(r)}$ the $r$-th derivative of $\pi$ as defined in \cite[\S 3.5 and \S 4.3]{MR0579172}. It is a functor from $\Alg(G_n)$ to $\Alg(G_{n-r})$.
As a consequence of \cite[Lemma 4.7(a)]{MR0579172} we have
\begin{lemma}\label{high_der1}
Let $\pi\in \Alg(G_{n})$. Then, $\supp(\pi^{(r)})\subseteq \supp(\pi)$ for all $0\leq r \leq n$. \qed
\end{lemma}

\subsubsection{}
As observed in \cite[(3.2)]{MR2417789}, it follows from \cite[Proposition 3.7]{MR584084} that for $n=2k+r$ and $\pi\in\Alg(G_n)$ there is a natural linear isomorphism
\begin{equation}\label{eq: der frob}
\Hom_{H_{2k,r}}(\pi,\psi)\simeq \Hom_{\Sp_{2k}(F)}(\pi^{(r)},1).
\end{equation}
This is the reason that we prefer $H_{2k,r}$ to $H_{r,2k}'$.

Note that, in particular, $\pi$ is generic if and only if $\pi^{(n)}\ne 0$.

\subsection{Hereditary property of Klyachko models}

As we observe bellow, Klyachko models behave well with respect to parabolic induction.
\begin{proposition}\label{prop: herad}
Let $\pi_i\in\Alg(G_{n_i})$ and $n_i=2k_i+r_i$ be such that $\pi_i$ is $(H_{2k_i,r_i},\psi)$-distinguished for $i=1,\dots,t$. Then $\pi=\pi_1\times\cdots\times\pi_t$ is $(H_{2k,r},\psi)$-distinguished where $k=k_1+\cdots+ k_t$ and $r=r_1+\cdots+r_t$.
\end{proposition}

\begin{proof}
Induction on $t$ reduces the statement to the case $t=2$ that we now assume. 
Let $\pi_s=\nu^s\pi_1\times\pi_2$ for $s\in \C$, so that $\pi=\pi_0$.
Recall that by the Leibnitz rule, \cite[Lemma 4.5]{MR0579172}, $\pi_s^{(r)}$ admits a filtration with factors $(\nu^s \pi_1)^{(i)}\times \pi_2^{(r-i)}\simeq\nu^s \pi_1^{(i)}\times \pi_2^{(r-i)}$, $i=0,\dots,r$.
Note, that there exists a small enough punctured neighborhood $U$ of $s=0$, so that for $i\ne j$ the central characters of the (finitely many) irreducible components of $\nu^s \pi_1^{(i)}\times \pi_2^{(r-i)}$ and of $\nu^s \pi_1^{(j)}\times \pi_2^{(r-j)}$ are disjoint. It follows that for $s\in U$ we have,
\[
\pi_s^{(r)}\simeq \oplus_{i=0}^r (\nu^s \pi_1^{(i)}\times \pi_2^{(r-i)}).
\]
In fact, when realizing all $\pi_s$ in the representation space of $\pi$, this direct sum decomposition is independent of $s\in U$. It follows, that there is a meromorphic map $P_s:\pi_s^{(r)}\rightarrow \nu^s \pi_1^{(r_1)}\times \pi_2^{(r_2)}$, that is surjective for $s\in U$.

By \eqref{eq: der frob}, $\pi_i^{(r_i)}$ is $\Sp_{2k_i}(F)$-distinguished, $i=1,2$ and therefore, by Lemma \ref{BD her}, there exists a non-zero, meromorphic family of linear forms $\ell_s$ such that in a possibly smaller punctured neighborhood $U_0$ of $s=0$ we have $\ell_s\in \Hom_{\Sp_{2k}(F)}(\nu^s \pi_1^{(r_1)}\times \pi_2^{(r_2)},1)$. Therefore, $\ell_s\circ P_s\in\Hom_{\Sp_{2k}(F)}(\pi_s^{(r)},1)$ is non-zero for $s\in U_0$. As in Corollary \ref{open dist}, a leading term argument implies that $\pi^{(r)}=\pi_0^{(r)}$ is $\Sp_{2k}(F)$-distinguished and therefore, by \eqref{eq: der frob}, $\pi$ is $(H_{2k,r},\psi)$-distinguished.

\end{proof}

\subsection{Reduction to cuspidal lines}
We reduce the study of $(H_{2k,r},\psi)$-distinguished representations in $\Alg$ to rigid representations, in fact, more generally to totally disjoint supports (Definition \ref{def: tot disj}). 
\begin{proposition}\label{lem: Kly cusp}
Let $\pi_i\in\Alg$, $i=1,\dots,t$ be such that $\supp(\pi_i)$ and $\supp(\pi_j)$ are totally disjoint for all $i\ne j$. Then $\pi=\pi_1\times\cdots\times\pi_t$ admits a Klyachko model if and only if $\pi_i$ admits a Klyachko model for all $i=1,\dots,t$.  
More precisely, $\pi$ is $(H_{2k,r},\psi)$-distinguished if and only if $\pi_i$ is $(H_{2k_i,r_i},\psi)$-distinguished, $i=1,\dots,t$ for some decomposition $r=r_1+\cdots+r_t$.
\end{proposition}
\begin{proof}
The `if' part is immediate from Proposition \ref{prop: herad}. Assume that $\pi\in\Alg(G_n)$, $n=2k+r$ and $\pi$ is $(H_{2k,r},\psi)$-distinguished.
By \eqref{eq: der frob}, $\pi^{(r)}$ is $\Sp_{2k}(F)$-distinguished. Therefore, by the Leibnitz rule, \cite[Lemma 4.5]{MR0579172}, and Lemma \ref{lem: dist comp}, there exists a decomposition $r=r_1+\cdots+r_t$ such that $\pi_1^{(r_1)}\times\cdots\times\pi_t^{(r_t)}$ is $\Sp_{2k}(F)$-distinguished. Since, by Lemma \ref{high_der1}, $\supp(\pi_i^{(r_i)})\subseteq\supp(\pi_i)$ it now follows from Lemma \ref{lem: useful cusp} that $\pi_i^{(r_i)}$ is $\Sp$-distinguished for all $i=1,\dots,t$. The lemma now follows from \eqref{eq: der frob}.
\end{proof}

\section{Klyachko models for ladder representations}\label{s: kl ladder}

We classify all ladder representations that admit, any given, Klyachko model.

\subsection{Klyachko models for Proper ladders}
\subsubsection{Proper Ladders}\label{sss: prop l}
\begin{definition} A ladder, $\mult=\{\Delta_1,\dots,\Delta_k\}\in\OO_\rho$ is called a \emph{proper ladder} if $\Delta_{i+1}\prec\Delta_i$, $i=1,\dots, k-1$. If $\mult$ is a proper ladder then $L(\mult)$ is called a proper ladder representation.
\end{definition}
In fact, if $\mult\in\OO_\rho$ is a proper ladder then $\mult^t$ is also a proper ladder, hence $Z(\mult)$ is a proper ladder representation, but this fact will not be used in the sequel.

\subsubsection{}\label{sss: prop}
Note that if $\mult\in\OO_\rho$ is a ladder then it can be decomposed uniquely (up to order) as a sum $\mult=\mult_1+\cdots+\mult_t$ where $\mult_i$ is a proper ladder for all $i=1,\dots,t$ and $\Delta\not\prec\Delta'$ for all $i\ne j$, $\Delta\in\mult_i$ and $\Delta'\in \mult_j$. Therefore, $\supp(L(\mult_i))$ and $\supp(L(\mult_j))$ are totally disjoint for all $i\ne j$ and, by Lemma \ref{lem; tot disj is irr},
\[
L(\mult)=L(\mult_1)\times\cdots\times L(\mult_t).
\]
In other words, any ladder representation is a product of proper ladder representations uniquely determined up to order.


\subsubsection{Right aligned segments} We define the following relation on segments of cuspidal representations.
\begin{definition}\label{def: ra seg}
For segments $\Delta=[a,b]_{(\rho)}$ and $\Delta'=[a',b']_{(\rho)}$ we say that $\Delta'$ is right-aligned with $\Delta$ and write $\Delta' \vdash \Delta$ if 
\begin{itemize}
\item $a\ge a'+1$ and
\item $b= b'+1$.
\end{itemize}
We label this relation by the integer $r=d(a-a'-1)$ where $\rho\in \Alg(G_d)$ and write $\Delta'\vdash_r \Delta$.
\end{definition}
Note, in particular, that $\Delta'\vdash_0\Delta$ means that $\Delta=\nu\Delta'$.
\begin{example}
Let $\Delta=[4,7]_{(\rho)}$ and $\Delta'=[0,6]_{(\rho)}$ be segments. Then 
\[
\tiny{\xymatrix{
&&&&&&&\overset{4}\circ&\overset{5}\circ\ar@{-}[l]&\overset{6}\circ\ar@{-}[l]&\overset{7}\circ\ar@{-}[l] &\\
&&&\overset{0}\circ&\overset{1}\circ\ar@{-}[l]&\overset{2}\circ\ar@{-}[l]&\overset{3}\circ\ar@{-}[l] &\overset{4}\circ\ar@{-}[l]&\overset{5}\circ\ar@{-}[l]&\overset{6}\circ\ar@{-}[l] & }}
\]
illustrates the relation $\Delta'\vdash_{3d}\Delta$ if $\rho\in\Alg(G_d)$. 

\end{example}
\subsubsection{} Before characterizing the proper ladder representations admitting Klyachko models we need the following technical result.
\begin{lemma}\label{lem: left-alg}
Let $d$ be such that $\rho\in \Alg(G_d)$ and $\mult=\{\Delta_1,\dots,\Delta_t\}\in \OO_\rho$ a proper ladder. Write $\Delta_i=[a_i,b_i]_{(\rho)}$. Suppose that $c_1>\cdots>c_t$ are integers such that $a_i-1\le c_i\le b_i$, $i=1,\dots,t$ and let $\mult_1=\{[c_1+1,b_1]_{(\rho)},\dots,[c_t+1,b_t]_{(\rho)}\}$ and $\mult_2=\{[a_1,c_1]_{(\rho)},\dots,[a_t,c_t]_{(\rho)}\}$ be the associated ladders.  If either $\mult_1=0$ or $L(\mult_1)$ is $\Sp$-distinguished and either $\mult_2=0$ or $L(\mult_2)$ is generic then $c_{t-2i}+1=c_{t-2i-1}=a_{t-2i-1}-1$ and $\Delta_{t-2i}\vdash_{r_i}\Delta_{t-2i-1}$ where $r_i=d(a_{t-2i-1}-a_{t-2i}-1)$ for all $i=0,\dots,\lfloor t/2\rfloor-1$. Moreover, if $t$ is odd then $c_1=b_1$. 
\end{lemma}
\begin{proof}
If $t=1$ then the lemma follows from the fact that if $\mult_1\ne 0$ then $L(\mult_1)$ is generic and Lemma \ref{gen not sp}. 

Assume that $t>1$.
Suppose that $1< i\le t$ is such that $c_i=b_i$ (in particular, $[a_{i},c_{i}]_{(\rho)}$ is not empty). Then, since $\mult$ is a proper ladder, we have $a_{i-1}-1\le b_i=c_i<c_{i-1}$ and therefore $[a_{i-1},c_{i-1}]_{(\rho)}$ is non-empty. But then $[a_i,c_i]_{(\rho)}\prec [a_{i-1},c_{i-1}]_{(\rho)}$ are both in $\mult_2$. By \cite[Theorem 9.7]{MR584084} this contradicts the assumption that $L(\mult_2)$ is generic. Therefore, $c_i<b_i$ for all $i>1$.

By the assumption that $L(\mult_1)$ is $\Sp$-distinguished and Theorem \ref{thm: dist lad}, $\mult_1$ is of Speh type. That is,  $c_1<b_1$ if and only if $t$ is even and either way, 
\[[c_{t-2i-1}+1,b_{t-2i-1}]_{(\rho)}=\nu[c_{t-2i}+1,b_{t-2i}]_{(\rho)}, \ \ \ i=0,\dots,\lfloor t/2\rfloor-1. 
\]
To complete the proof it is only left to show that $c_{t-2i-1}=a_{t-2i-1}-1$ for all $i=0,\dots,\lfloor t/2\rfloor-1$. But if $c_{t-2i-1}\ge a_{t-2i-1}$ then $c_{t-2i}=c_{t-2i-1}-1\ge a_{t-2i-1}-1$, i.e., $[a_{t-2i},c_{t-2i}]_{(\rho)}\prec [a_{t-2i-1},c_{t-2i-1}]_{(\rho)}$ in $\mult_2$ which, again by \cite[Theorem 9.7]{MR584084}, is a contradiction. The lemma follows. 
\end{proof}

\subsubsection{}
We now determine the proper ladder representations that admit any particular Klyachko model.
\begin{proposition}\label{prop: prop lad}
Let $\mult=\{\Delta_1,\dots,\Delta_t\}\in\OO_\rho$ be a proper ladder, so that $L(\mult)\in \Alg(G_n)$ and let $n=2k+r$. If $t$ is odd, let $s$ be such that $L(\Delta_1)\in \Alg(G_s)$, otherwise, set $s=0$. Then $L(\mult)$ is $(H_{2k,r},\psi)$-distinguished if and only if $\Delta_{t-2i}\vdash_{r_i}\Delta_{t-2i-1}$ for some $r_i$, $i=0,\dots,\lfloor t/2\rfloor-1$ and $r=r_0+\cdots +r_{\lfloor t/2\rfloor-1}+s$. 
\end{proposition}
\begin{proof}
Let $\pi=L(\mult)$ and note that $\Ind_{\Sp_{2k}(F) \times N_{r}}^{M_{(2k,r)}}(1\otimes \psi)=\Ind_{H_{2k,r}}^{P_{(2k,r)}}(\psi)|_{M_{(2k,r)}}$.  By \eqref{eq: frob rec}, \eqref{eq: 1st adj} and transitivity of induction we have
\begin{multline}\label{eq: frobs}
\Hom_{H_{2k,r}}(\pi,\psi)\simeq\Hom_{G_n}(\pi,\Ind_{H_{2k,r}}^{G_n}(\psi)) \simeq \\  \Hom_{M_{(2k,r)}}(\jm_{M_{(2k,r)},G_n}(\pi),\Ind_{\Sp_{2k}(F) \times N_{r}}^{M_{(2k,r)}}(1\otimes\psi))\simeq \Hom_{\Sp_{2k}(F)\times N_r}(\jm_{M_{(2k,r)},G_n}(\pi),1\otimes \psi).
\end{multline}

Assume first that $\pi$ is $(H_{2k,r},\psi)$-distinguished. By \eqref{eq: frobs} and Lemma \ref{lem: dist comp} there is an irreducible component $\sigma_1\otimes\sigma_2$ of $\jm_{M_{(2k,r)},G_n}(\pi)$, (where $\sigma_1\in\Alg(G_{2k})$ and $\sigma_2\in \Alg(G_r)$) so that $\sigma_1$ is $\Sp$-distinguished and $\sigma_2$ is generic. 

If $\Delta_i=[a_i,b_i]_{(\rho)}$ then it follows from \cite[Theorem 2.1]{MR2996769} that there exist $c_1>\cdots>c_t$ such that $\sigma_1=L(\mult_1)$ and $\sigma_2=L(\mult_2)$ where $\mult_1=\{[c_1+1,b_1]_{(\rho)},\dots,[c_t+1,b_t]_{(\rho)}\}$ and $\mult_2=\{[a_1,c_1]_{(\rho)},\dots,[a_t,c_t]_{(\rho)}\}$. The `only if' part of the proposition therefore follows from Lemma \ref{lem: left-alg}. 

Assume that $\Delta_{t-2i}\vdash_{r_i}\Delta_{t-2i-1}$, $i=0,\dots,\lfloor t/2\rfloor-1$ and $r=r_0+\cdots +r_{\lfloor t/2\rfloor-1}+s$.
Let $c_{t-2i}+1=c_{t-2i-1}=a_{t-2i-1}-1$, $i=0,\dots,\lfloor t/2\rfloor-1$. If $t$ is odd, further let $c_1=b_1$. Let
$\sigma_1=L(\mult_1)$ and $\sigma_2=L(\mult_2)$ where $\mult_1=\{[c_1+1,b_1]_{(\rho)},\dots,[c_t+1,b_t]_{(\rho)}\}$ and $\mult_2=\{[a_1,c_1]_{(\rho)},\dots,[a_t,c_t]_{(\rho)}\}$. Note that $\Delta\not\prec\Delta'$ for any two segments in the ladder $\mult_2$ and therefore $\sigma_2$ is generic by \cite[Theorem 9.7]{MR584084}. It is also clear from the above definitions that $\mult_1$ is of Speh type and therefore $\sigma_1$ is $\Sp$-distinguished by 
Theorem \ref{thm: dist lad}. By \cite[Corollary 2.2]{MR2996769}, $\sigma_1\otimes\sigma_2$ is a direct summand (and in particular a quotient) of $\jm_{M_{(2k,r)},G_n}(\pi)$. Therefore \eqref{eq: frobs} and Lemma \ref{drmk: ist quot} complete the proof of the proposition.
\end{proof}

\begin{remark}
If $\pi$ is a proper ladder representation then Proposition \ref{prop: prop lad} provides a recipe for computing $r(\pi)$ and in particular, directly implies the uniqueness of $r(\pi)$. The same is true more generally for ladder representations.
\end{remark}

\subsection{Klyachko models of ladder representations}

\begin{theorem}\label{thm: kly lad}
Let $\pi$ be a ladder representation and assume that $\pi=\pi_1\times\cdots\times\pi_t$ is the unique decomposition of $\pi$ as a product of proper ladder representations (see \S \ref{sss: prop}). Then $\pi$ admits a Klyachko model if and only if $\pi_i$ admits a Klyachko model for all $i=1,\dots,t$. Furthermore, in that case $r(\pi)=r(\pi_1)+\cdots+r(\pi_t)$.
\end{theorem}
\begin{proof}
This is immediate from Proposition \ref{lem: Kly cusp} and \S \ref{sss: prop}.
\end{proof}

\begin{remark}
Based on \eqref{eq: 2 models}, the classification of ladder representations that are $(H_{r,2k}',\psi)$-distinguished is obtained by `reflecting all segments along the origin of their $\Z$-line'.
\end{remark}

\def\cprime{$'$} \def\cprime{$'$} \def\Dbar{\leavevmode\lower.6ex\hbox to
  0pt{\hskip-.23ex \accent"16\hss}D}
  \def\cftil#1{\ifmmode\setbox7\hbox{$\accent"5E#1$}\else
  \setbox7\hbox{\accent"5E#1}\penalty 10000\relax\fi\raise 1\ht7
  \hbox{\lower1.15ex\hbox to 1\wd7{\hss\accent"7E\hss}}\penalty 10000
  \hskip-1\wd7\penalty 10000\box7}
  \def\cfudot#1{\ifmmode\setbox7\hbox{$\accent"5E#1$}\else
  \setbox7\hbox{\accent"5E#1}\penalty 10000\relax\fi\raise 1\ht7
  \hbox{\raise.1ex\hbox to 1\wd7{\hss.\hss}}\penalty 10000 \hskip-1\wd7\penalty
  10000\box7} \def\cftil#1{\ifmmode\setbox7\hbox{$\accent"5E#1$}\else
  \setbox7\hbox{\accent"5E#1}\penalty 10000\relax\fi\raise 1\ht7
  \hbox{\lower1.15ex\hbox to 1\wd7{\hss\accent"7E\hss}}\penalty 10000
  \hskip-1\wd7\penalty 10000\box7} \def\cprime{$'$}
  \def\Dbar{\leavevmode\lower.6ex\hbox to 0pt{\hskip-.23ex \accent"16\hss}D}
  \def\cftil#1{\ifmmode\setbox7\hbox{$\accent"5E#1$}\else
  \setbox7\hbox{\accent"5E#1}\penalty 10000\relax\fi\raise 1\ht7
  \hbox{\lower1.15ex\hbox to 1\wd7{\hss\accent"7E\hss}}\penalty 10000
  \hskip-1\wd7\penalty 10000\box7}
  \def\polhk#1{\setbox0=\hbox{#1}{\ooalign{\hidewidth
  \lower1.5ex\hbox{`}\hidewidth\crcr\unhbox0}}} \def\dbar{\leavevmode\hbox to
  0pt{\hskip.2ex \accent"16\hss}d}
  \def\cfac#1{\ifmmode\setbox7\hbox{$\accent"5E#1$}\else
  \setbox7\hbox{\accent"5E#1}\penalty 10000\relax\fi\raise 1\ht7
  \hbox{\lower1.15ex\hbox to 1\wd7{\hss\accent"13\hss}}\penalty 10000
  \hskip-1\wd7\penalty 10000\box7}
  \def\ocirc#1{\ifmmode\setbox0=\hbox{$#1$}\dimen0=\ht0 \advance\dimen0
  by1pt\rlap{\hbox to\wd0{\hss\raise\dimen0
  \hbox{\hskip.2em$\scriptscriptstyle\circ$}\hss}}#1\else {\accent"17 #1}\fi}
  \def\bud{$''$} \def\cfudot#1{\ifmmode\setbox7\hbox{$\accent"5E#1$}\else
  \setbox7\hbox{\accent"5E#1}\penalty 10000\relax\fi\raise 1\ht7
  \hbox{\raise.1ex\hbox to 1\wd7{\hss.\hss}}\penalty 10000 \hskip-1\wd7\penalty
  10000\box7} \def\lfhook#1{\setbox0=\hbox{#1}{\ooalign{\hidewidth
  \lower1.5ex\hbox{'}\hidewidth\crcr\unhbox0}}}
\providecommand{\bysame}{\leavevmode\hbox to3em{\hrulefill}\thinspace}
\providecommand{\MR}{\relax\ifhmode\unskip\space\fi MR }
\providecommand{\MRhref}[2]{%
  \href{http://www.ams.org/mathscinet-getitem?mr=#1}{#2}
}
\providecommand{\href}[2]{#2}


\end{document}